\documentclass[9pt]{amsart}
\usepackage[pagebackref,linktocpage=true,colorlinks=true,linkcolor=Blue,citecolor=BrickRed,urlcolor=RoyalBlue]{hyperref}
\usepackage[alphabetic]{amsrefs}

\usepackage[margin=1in, width=8in]{geometry}

\usepackage[usenames,dvipsnames,svgnames,table]{xcolor}

\usepackage{amsmath,amssymb,amsfonts,amsthm,enumerate,textcomp}

\usepackage{relsize}

\usepackage{mathrsfs}

\usepackage{tikz}
\usetikzlibrary{arrows,shapes,positioning}
\usetikzlibrary{patterns}
\usetikzlibrary{decorations.markings}
\tikzstyle arrowstyle=[scale=1]
\tikzstyle directed=[postaction={decorate,decoration={markings,
    mark=at position .65 with {\arrow[arrowstyle]{stealth}}}}]
\tikzstyle reverse directed=[postaction={decorate,decoration={markings,
    mark=at position .65 with {\arrowreversed[arrowstyle]{stealth};}}}]

\newtheorem{theorem}{Theorem}                                                                                %
\newtheorem{lemma}{Lemma}                                                                                     %
\newtheorem{proposition}{Proposition}                                                                        %
\newtheorem{definition}{Definition}                                                                               %
\newtheorem{remark}[theorem]{Remark}                                                                      %
\renewcommand\epsilon\varepsilon 
\newcommand{\p}{\partial}

\newcommand\R{\mathbb R} 
\newcommand\na{\nabla}

\newcommand\dist{\text{dist}}
\renewcommand\det{\text{det}}

\renewcommand\phi{\varphi}

\newcommand\U{\mathscr U}
\newcommand\V{\mathscr V}
\renewcommand\S{\mathbb S}
\newcommand\id{\text{Id}}
\newcommand\A{\mathpzc D}

\newcommand\Ov{\mathcal O}
                                                   
\numberwithin{equation}{section} 

\DeclareFontFamily{OT1}{pzc}{}                                                          %
\DeclareFontShape{OT1}{pzc}{m}{it}                                                   %
             {<-> s * [1.250] pzcmi7t}{}                                                       %
\DeclareMathAlphabet{\mathpzc}{OT1}{pzc}                                       %
                                 {m}{it}                                                                  %

\renewcommand{\H}[1]{\bf H{#1}^\circ}

\thanks{AMS Classifications: Primary 35J96, Secondary 78A05, 78A46, 78A50.\\ 
Keywords: Monge-Amp\`ere type equations, local regularity, MTW condition, far field refractor, antenna design.}

\begin{document}
\title{Refractor surfaces determined by near-field data}

\author{Aram L. Karakhanyan}
\address{Aram Karakhanyan:
School of Mathematics, The University of Edinburgh,
Peter Tait Guthrie Road, EH9 3FD Edinburgh, UK}
\email{aram6k@gmail.com}

\author{Ahmad Sabra}
\address{Ahmad Sabra:
American University of Beirut, P.O.Box 11-0236 /
Riad El-Solh / Beirut 1107 2020; and University of Warsaw Krakowskie Przedmie\'scie 26/28. 00-927 Warsaw. Poland}
\email{asabra@aub.edu.lb}

\maketitle


\begin{abstract}
 In this paper we study the near-field refractor problem with point source at the origin 
 and prescribed target on the given receiver surface $\Sigma$. This nonvariational  problem 
 can be studied in the framework of prescribed Jacobian equations. We construct the corresponding generating function and show that the Aleksandrov and the Brenier type solutions are equivalent. 
Our main result establishes local smoothness of  Aleksandrov's solutions when the data is smooth and when the medium containing the source has smaller refractive index than the medium containing the target. This is done by deriving  the Monge-Amp\`ere type equation that smooth solutions satisfy and establishing the validity of  the MTW condition
 for a large class of receiver surfaces,  which in turn implies the local $C^2 $ regularity of the 
 refactor.  
\end{abstract}
{\small{ \tableofcontents}}

\section{Introduction}

We consider the following setting. $\S^n$ is the unit sphere in $\R^{n+1}, n\ge 2$, and $\U$ an open subset of the upper hemisphere $\S^n_+$. Let $\Sigma$ be a regular hypersurface in $\R^{n+1}$ given implicitly by a function $\psi$ and $\V$ an open set in $\Sigma.$
We are given two smooth positive function $f\in L^1(\mathcal U)$ and $g\in L^1(\mathcal V)$ such that
\begin{equation}\label{eq:energy conservation}
\int_{\U} f\, d\sigma(\S^n)=\int_{\V} g\, d\sigma(\Sigma).
\end{equation}

Assume $\Gamma$ is a refractive surface parametrized radially over $\U$, $\Gamma=\{\rho(x)X: X=(x,x_{n+1})\in \U\},$ and separating two homogeneous media with respective refractive indices $n_1$, $n_2$. In medium $n_1$, rays are emitted from the origin $O$ with direction $X\in\U$, and are refracted into medium $n_2$ with the unit direction $Y$ according to the Snell's law given in \eqref{eq:Snell}, see Figure \ref{fig1}. The refracted ray strikes $\Sigma$ at the point $Z(x)$. We say that $\Gamma$ solves the near field refractor problem if $Z(x)\in \V$ for every $(x,x_{n+1})\in \U$, and
\begin{equation}\label{en-loc-1}
\int_E f(X)\, d\sigma(\S^n_+)=\int_{\mathcal R(E)} g(Z)\, d\sigma(\Sigma),\ \mbox{for the Borel subsets}\ E\subset \U
\end{equation}
with $\mathcal R$ the imaging map corresponding to $\Sigma$, defined as follows:  
$$Q\in \mathcal R(E)\iff Q=Z(x) \text{ for some } (x,x_{n+1})\in E.$$

Let $Y(x)$ be the unit direction of the refracted ray 
corresponding to $X=(x, x_{n+1})$.
We can write $Z(x)=\rho(x)X+tY$ where $t(x)=|Z(x)-\rho(x)X|$ is the distance that the 
refracted ray travels before striking $\Sigma$ at $Z(x)$. All over the paper, we assume the following geometric conditions on the solution $\rho$, the source $\U$ and the target $\Sigma$:
\begin{itemize}
\item[$\H1$] $y_{n+1}>0$. This condition requires the refractor to be below the target. 
\item[$\H2$] $\psi^{n+1}>0$, with $\psi^{n+1}:=\dfrac{\partial \psi}{\partial z_{n+1}}.$
\item[$\H3$] $\na \psi(Z(x))\cdot Y(x)>0$, i.e. the target is visible  in  the direction $Y$.
%
%
\item[$\H4$] $\dist (\U, \V)>0$.
\end{itemize}

For every $X=(x,x_{n+1})\in \U$, we write $Z(x)=(z_1(x),\cdots,z_{n}(x),z_{n+1}(x)):=(z(x),z_{n+1}(x))$. Notice that if $\Gamma$ solves the near field refractor problem then
$$ \dfrac{f}{g}=\dfrac{d\sigma(\Sigma)}{d\sigma(\S^n_+)},$$
where $d\sigma(\Sigma)=\dfrac{\left|\nabla \psi\right|}{\psi^{n+1}} |\det z(x)|dz,$ and $d\sigma(\S^n_+)=\dfrac{dx}{\sqrt{1-|x|^2}}=\dfrac{dx}{x_{n+1}}$. Consequently  
\begin{equation}\label{compact}
|\det Dz(x)|=\dfrac{f(X)\,\psi^{n+1}(Z(x))}{|\nabla \psi(Z(x))| g(Z(x)) x_{n+1}}.
\end{equation}

The existence of a Brenier solution to the far field refractor problem is studied in \cite{Gut-Hua}, i.e. in this case $\Gamma$ is such that for every Borel set $F\in \Sigma$ 
$$\int_{\mathcal R^{-1}(F)} f=\int_F g.$$
Solutions are constructed so that they are supported at every point by a cartesian oval. 

In order to study the regularity of weak solutions one has to rewrite the equation $|\det Dz|=\frac fg\frac{|\psi^{n+1}|}{|\na \psi|}$ 
as a Monge-Amp\`ere type equation $\det (D^2\rho +H(x, \rho, \na \rho))=h(x, \rho, \na \rho)$ with some matrix $H:=H(x,v,p), x\in \R^n, v\in \R, p\in \R^n$ and function $h$.  Note that the matrix
$D^2\rho +H(x, \rho, \na \rho)$ is positive definite thanks to the existence of touching ovaloids from below if $\kappa<1$.   
Moreover,  the matrix $H$
is to verify the Ma-Trudinger-Wang condition
 \cite{MTW} 
\begin{equation}\label{def-MTW}
\frac{\partial^2 H_{ij}}{\p p_k\p p_l}\xi_i\xi_j\eta_k\eta_l\leq -c_0|\xi|^2|\eta|^2, \quad \eta, \xi\in \R^n, \xi\perp\eta, 
\end{equation}
with some positive constant $c_0>0$. This condition alone does not immediately imply that the generalized solutions are locally  smooth, and some further structural conditions must be imposed. 



The paper is organized as follows. In Section \ref{sec:2}, we use the Snell's law and the stretch function $t(x)=\left|Z(x)-\rho(x)X\right|$ to write \eqref{compact} into an equation of Monge-Amp\`ere type and prove the following theorem:

\begin{theorem}\label{thm-1}
Let $X=(x, x_{n+1}), x_{n+1}=\sqrt{1-|x|^2}, |x|<1$ be  the canonical parametrization of the upper hemisphere  and regard
 $\rho $  as a $C^2$ function of $x$. 
 If $\Gamma=\{X\in\R^{n+1}\,  :\, \rho(x)X, \, X\in\U\}$ is a solution to the near field refractor problem satisfying conditions $\H1-\H4$ then $\rho$ solves the following
 \begin{equation}\label{eq-intro-h}
\det (\mathcal M)\, \det
 \left(
 D^2\rho
 + \dfrac{1}{bt}\mathcal A+\dfrac{1}{b}\mathcal B
 \right)
 = 
\frac {f}{g}\frac{y_{n+1}}{x_{n+1}} \dfrac{\nabla \psi\cdot Y}{|\nabla \psi|}\frac1{t^n}, 
 \end{equation}
where the matrices $\mathcal M, \mathcal A, \mathcal B$ are given respectively in  \eqref{eq:Formula for M}, \eqref{eq:mathcal A}, \eqref{eq:mathcal B}. Moreover,  the matrix $\mathcal M$ is invertible in some small  neighborhood of $x=0$ .
Here $b(x,\rho,\nabla \rho)=\dfrac{\kappa \rho-\sqrt{\rho^2+(1-\kappa^2)(|\nabla \rho|^2-(\nabla \rho\cdot x)^2)}}{\rho^2+|\nabla \rho|^2-(\nabla \rho\cdot x)^2}$, and  $\kappa=\dfrac{n_1}{n_2}$.
\end{theorem}

We mention that using a different stretch function, a Monge Amp\`ere type PDE for solutions to the near field problem is derived in \cite[Appendix]{Gut-Hua}, however the equation is not simplified and the invertibility of the matrices involved in the equation requires further checking.

In section \ref{sec:Simplifying}, we simplify further the matrices $\mathcal A$ and $\mathcal B$, and then calculate \eqref{eq-intro-h} at $x=0$ yielding to the following Theorem.

\begin{theorem}\label{thm:MA at 0}
At $x=0$, we have
\begin{align}  \label{eq:A at 0}
\mathcal A(0,\rho,\nabla\rho)&=\rho Id -\dfrac{\kappa^2-1}{\rho} \nabla \rho\otimes\nabla \rho, \\ \label{eq:B at 0}
\mathcal B(0,\rho,\nabla \rho)&=(\kappa-b\rho) Id -\dfrac{2b}{\rho}\nabla \rho\otimes\nabla \rho.
\end{align}

Therefore \eqref{eq-intro-h} becomes at $x=0$
\begin{equation}\label{eq:PDE at 0}
\det\left(D^2\rho+\dfrac{1}{b t} \left(\rho Id -\dfrac{\kappa^2-1}{\rho} \nabla \rho\otimes\nabla \rho\right)+\dfrac{\kappa-b \rho}{b} Id -\dfrac{2}{\rho}\nabla \rho\otimes\nabla \rho\right)=
\dfrac{f}{g\,\det (\mathcal M)}\dfrac{\nabla \psi\cdot Y}{\nabla \psi}\dfrac{1}{t^n}y_{n+1}.
\end{equation}
\end{theorem}

We stress the fact that Theorems \ref{thm-1}, and \ref{thm:MA at 0} are valid independently of whether $n_1$ is smaller or larger than $n_2$.

Writing  \eqref{eq:PDE at 0} of the form $\det (D^2\rho+I+II)=h$, we calculate in Section \ref{sec:MTW}, the MTW condition and show that for orthogonal vectors $\xi$, $\eta$
\begin{eqnarray}\label{MTW-0}
\sum_{i,j,l,m}\left[\frac{\partial^2 \left( I_{ij}+II_{ij}\right)}{\partial p_l\partial p_m}\right] \xi_i\xi_j\eta_l\eta_m
&=&
\sum_{l,m}
H_{lm}\eta_l\eta_m.
\end{eqnarray}
with $H$ given in \eqref{eq:form of H_{lm}}. Studying the sign of $H$ we conclude in Section \ref{sec:MTW} the following Theorem.

\begin{theorem}\label{thm-2}
The MTW-condition $\sum_{lm} H_{lm} \eta_l\eta_m<0,$  is satisfied at $x=0$, provided that 
$\kappa<1$, hypotheses $\H1-\H4$ hold, and $\psi$ is concave in $Y$ direction, i.e.
 $\Sigma$ is a concave graph in $Y$ direction near  $Z$, see Figure \ref{fig1}. 

\end{theorem}

Theorem \ref{thm-2} implies the MTW condition at every point. In fact, let $Z_0\in \Sigma$. Rotate the coordinate system of $\R^{n+1}$ in such way that $Z$ is the image of the north pole in the new
coordinate system. Since $\mathcal A, \mathcal B$ are smooth functions of their arguments it follows that if $H<0$ at 
$x=0$ then $H< 0$ in some neighborhood  of $x=0$. We also briefly discuss the case $\kappa>1$ in Remark \ref{rem:zagalla}.

In Section \ref{sec:Regularity}, we  prove the $C^{2,\alpha}$ regularity the Brenier solutions constructed in \cite{Gut-Hua}. To do this  we first prove in Section \ref{sec:Aleksandrov}, using a Legendre type transform, that these refractors are Aleksandrov solution, i.e. satisfy \eqref{en-loc-1}. We next show in Section \ref{sec:loc glob}, that if $\V$ is R-convex  (see Definition \ref{def-R-conv}) a local support ovaloid of  the weak solution is in fact a global support ovaloid, which follows from the MTW property of the solution and the visibility conditions assumed in \cite{Gut-Hua}, see Lemma \ref{lem:loc-glob}. Using this geometric property, one can conclude from 
the a priori estimate established in  \cite{MTW}, \cite{Trudinger-D}, \cite{Trudinger}, the local 
$C^{2,\alpha}$ regularity by using the standard mollification argument 
in small balls, see also \cite{KW}. Thus we have 
\begin{theorem}\label{thm:thm4}
Let $f, g\in C^2$ and $\lambda\le f, g\le\Lambda$ for some constants $\lambda, \Lambda>0$. Let $p\in \Gamma\cap \{tX, X\in \U,  t>0\}$
such that $\H1-\H4$ are satisfied, for every such $p$, and $\Sigma$ is $R$-convex (see Definition \ref{def-R-conv}). Then every Aleksandrov weak solution of the refractor problem is locally $C^2$ smooth.  
\end{theorem}
The conditions in Theorem \ref{thm:thm4} are optimal, i.e.  if one of the conditions stated above fails then examples of non $C^2$ solutions can be constructed as in \cite{KW}.

We mention that existence and regularity of other refractor problems is already studied in the literature. The far field case when the target is a set of direction in the sphere is formulated in \cite{Gut-Hua-2} as a mass transport problem and the MTW conditions for the cost functions was analyzed in both cases when $n_1<n_2$ and when $n_1>n_2$. In the case of a planar source, i.e. when radiations are emanated from a plane with vertical direction, $C^{1,\alpha}$ estimates is proved in \cite{Gut-Tou} when $n_1<n_2$, and in \cite{Abe-Gut-Giu} when $n_1>n_2$. Local $C^2$ regularity of the planar source refractor problem is studied in \cite{K-siam}. 
Existence of the solutions with lenses (i.e. two surfaces) and arbitrary radiant field is proved in \cite{Gut-Sab:SIAM}, \cite{Gut-Sab:Freeform}. 

Near field problems differs point from the far field case since it cannot be formulated as a mass transport setting. \cite{Trudinger},  formulated the problem into a prescribed Jacobian setting and formulates conditions for existence and regularity of these solutions. More results in this direction were obtained in \cite{Trudinger-D2}, \cite{Gut-Abe}, and \cite{Jun-Gui}. In these paper, the authors analyzed sufficient conditions that shall be satisfied by surfaces solutions to Monge-Amp\`ere type equation in order to obtain smoothness of the optical surfaces. 

Energy problems for reflective surfaces have been also studied. In this case, Snell's law is much less complicated yielding to a simpler Monge-Amp\`ere type equation. Regularity is established in various papers. We mentions the following results for the far field case \cite{Wan}, \cite{Loep}, and \cite{KW} for the near field case.

We summarize the notation used in this paper in the following table:

\textbf{List of Notations}
\begin{align*}
\S^n_+&=\{X=(x,x_{n+1})\in \R^{n+1}:|x|^2+x_{n+1}^2=1, x_{n+1}>0\},\\
\kappa&=\dfrac{n_1}{n_2},\\
a&=\rho^2+|\nabla \rho|^2-(\nabla \rho\cdot x)^2,\\
b&=\dfrac{\kappa \rho-\sqrt{a-\kappa^2(a-\rho^2)}}{a},\\
q&=\sqrt{a-\kappa^2(a-\rho^2)},\\
\alpha&=-\dfrac{b^2}{\kappa^2-1}\left[\dfrac{\kappa q+\rho+(\kappa^2-1)x\cdot \nabla \rho}{q}\right],\\
\beta&=\dfrac{b^2}{q},\\
Q&=\dfrac{b}{q+b\left(|\nabla q|^2-\nabla\rho\cdot x(\rho+\nabla \rho\cdot x)\right)},\\
\sigma&=\kappa-b(\rho+x\cdot \na \rho),\\
\gamma &=b-Q\left(\sigma\,x\cdot \nabla \rho+b|\nabla \rho|^2\right),\\
F&=\sigma+\alpha |\nabla \rho|^2-(\alpha(\rho+x\cdot \nabla \rho)+2b)x\cdot \nabla \rho.\\
\end{align*}

\section{Preliminaries  and main formulae}\label{sec:2}
Let $\U$ be an open subset contained in $\S^n_+$. Let $\Gamma$ be the surface given radially by $\rho(x)X$, with $x=(x_1,\cdots,x_n)$, and $X=(x,x_{n+1})\in \U$, and $\rho$ is $C^2$ positive function. We assume that the surface $\Gamma$ separates two media $1$ and $2$ with corresponding constant refractive indices $n_1$ and $n_2$.
In medium $2$ we are given a target $\Sigma$ parametrized implicitly by the function $\psi$, i.e. $\psi(z_1,\cdots,z_{n+1})=0$ for every $(z_1,\cdots,z_{n+1})\in \Sigma$. Let $\V$ be an open subset of $\Sigma$.
We are given positive functions $f\in L^1(\U)$, and $g\in L^1(\V)$ such that
$$\int_{\U} f=\int_{\V} g.$$
A ray emitted from the source $O$ with unit direction $X\in \U$ is refracted at $\rho(x)X$ by $\Gamma$ into medium $2$ with unit direction $Y$, the refracted ray hits the target $\Sigma$ at the point $Z(x)=(z_1(x),\cdots,z_{n+1}(x))$.  Let $\mathcal R$ be the imaging map for the refractor $\Gamma$ defined as follows: for every Borel set $E\subseteq \U$
$$\mathcal R(E)=\{Q\in \Sigma: Q=Z(x) \text{ for some } (x,\sqrt{1-|x|^2})\in \U\}.$$
We say that $\Gamma=\{\rho(x)X \ : \ X\in \U\}$ is an Aleksandrov solution to the near field refractor problem if $Z(x)\in \V$ for every $(x,x_{n+1})\in \U$, and for every Borel set $E$
$$\int_{E} f(X)d\sigma(\S^n_+)=\int_{\mathcal R(E)}g(Z)d\sigma(\Sigma).$$
The goal of this section is to compute the differential equation satisfied by $\rho$, and prove Theorem \ref{thm-1}. For this we assume that $\rho$, $\U$ and $\V$ satisfies conditions $\H1-\H4$.

Let $\nu$ be the unit normal at $\rho(x)X$ toward the medium $n_2$. From  \cite{KW} Proposition 2.1 or \cite[Lemma 8.1]{GM}, $\nu$ is given by
\begin{equation}\label{eq:unit normal}
\nu(x) =-\dfrac {(\nabla \rho,0)-X\left( \rho+\nabla \rho\cdot x\right) }
{\sqrt {\rho^2+|\na \rho|^2-(x\cdot \na \rho)^2}}.
\end{equation}
At $\rho(x)X$ the ray emitted from $O$ with direction $X$ is refracted into the medium $n_2$ with unit direction $Y=(y_1,\cdots,y_n,y_{n+1}):=(y,y_{n+1})$.
 By Snell's law on $\R^{n+1}$, we have that $X-\dfrac{n_2}{n_1} Y$ is parallel to $\nu$, i.e.
 \begin{equation}\label{eq:Snell}
 X-\dfrac{n_2}{n_1} Y=\lambda \nu,
 \end{equation}
 with $\lambda=X\cdot \nu -\dfrac{n_2}{n_1}\sqrt{1-\left(\dfrac{n_1}{n_2}\right)^2(1-(X\cdot \nu)^2)},$ see \cite{Gut-Hua-2}.

\begin{figure}
\begin{center}
 \begin{tikzpicture}

    \coordinate (O) at (0.15,0.6) ;
    \coordinate (A) at (0,4) ;
    \coordinate (B) at (0,-4) ;
    
       
       \filldraw[blue!25!,opacity=.3] (-4, 0.5) to[out=18,in=-197] (O)  to[out=-18,in=110](4, -3.2)--(4,-4)--(4, 6)--(-4,6); 

       \filldraw[blue!60!,opacity=.3] (-4, 0.5) to[out=18,in=-197] (O)  to[out=-18,in=110](4, -3.2)--(4,-4)--(-4, -4)--(-4,0.5); 
              
    \node[right] at (2,2) {Medium $2$};
    \node[left] at (-2,-2) {Medium $1$};
    \node[right] at (2,4.8) {$\mathlarger{\mathlarger \Sigma}$};
\node[right] at (-0.6,4.8){$Z$};
        \node[] at (1.7, -3.6)  {$O$};

    \draw[dash pattern=on5pt off3pt] (1,4) -- (-1,-4) ; 

    \draw[red,ultra thick,directed] (O) -- (90:5.1); 
    \draw[blue,directed,ultra thick] (-70:4.24)--(O) ;
       \draw[->, black,ultra thick] (O) -- (76:3.5);
           \node[->, black,ultra thick] at (66:3)  {$\nu$};

    \draw (0.7, 2.8) arc (70:96:1.4); 
    \draw (-0.4,-1.6) arc (250:300:1.4) ;
    \node[] at (280:1.3)  {$\theta_{1}$};
        \node[] at (330:1.0)  {$X$};
    \node[] at (81.8:2.4)  {$\theta_{2}$};
        \node[] at (95.8:2.8)  {$Y$};
  \draw[thick] (3.7, 3.45) arc (50:110:7.4); 
\end{tikzpicture}
\end{center}
\caption{Snell's law of refraction, see Luneburg \cite{Luneburg} page 65.}\label{fig1}
\end{figure}
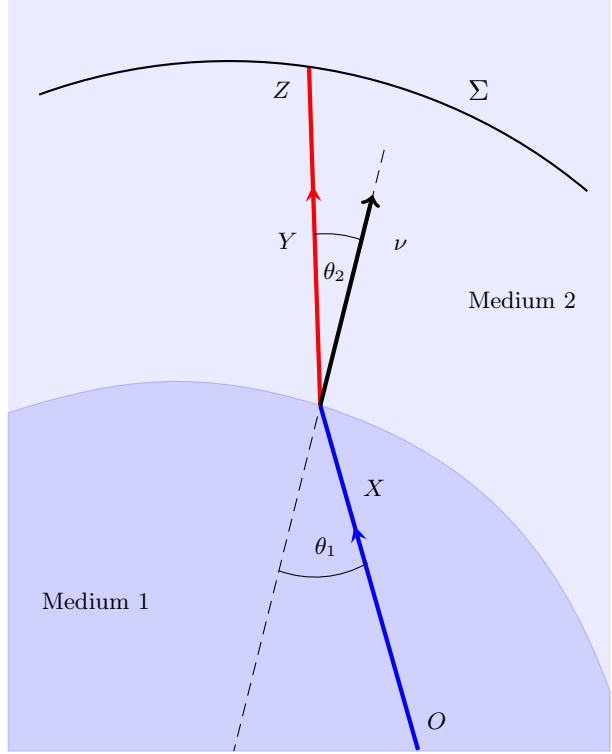

We have that the refracted  ray with direction $Y$ hits the surface $\Sigma$ at the point
\begin{equation}\label{def-Z}
Z(x)=\rho(x) X+t(x)Y(x):=(z(x),z_{n+1}(x)), 
\end{equation}
where $t(x)=|Z(x)-\rho(x)X|$ denotes the {\bf stretch function}.

\subsection{First form of $Dz$}
In this section we derive a formula for $Dz$ where 
$z$ is the projection of $Z$ onto $\{x_{n+1}=0\}$. We denote $\nabla \psi = \left(\dfrac{\p \psi}{\p z_1},\cdots, \dfrac{\p \psi}{\p z_{n}},\dfrac{\p \psi}{\p z_{n+1}}\right)=\left(\psi^1,\psi^2,\cdots,\psi^n,\psi^{n+1}\right)=\left(\widehat \nabla \psi ,\psi^{n+1}\right).$ Recall from Condition $\H1$, that $y_{n+1}>0$, and so $y_{n+1}=\sqrt{1-|y|^2}$.


\begin{proposition}\label{prop-Dz}
Let $\Sigma$ be defined implicitly as the zero level set of some smooth function $\psi:\R^{n+1}\to \R$.
Let $t$ be the stretch function determined implicitly from the identity $\psi(X\rho+tY)=0$. 
Let $z(x)=\rho(x)\,+t(x)y(x)=(z_1(x),\cdots,z_n(x))$ be the projection of $Z$, given by  \eqref{def-Z}, onto 
the plane $\{x_{n+1}=0\}$, then 
\begin{equation}\label{eq:Dz form 1}
Dz=(\partial_{x_j} z_i)_{i,j}=\mu_1+t\mu_2 Dy=\mu_2\left(\mu_0+t  Dy\right),\\
\end{equation}
where the matrices $\mu_1, \mu_2,\mu_0$ are defined below
 \begin{eqnarray*}
 \mu_1
 &=& 
 \rho\, \id +x\otimes \nabla \rho -\dfrac{1}{\nabla \psi \cdot Y}\left(\rho\, \left(y\otimes \widehat \nabla \psi\right)-\dfrac{\rho\psi^{n+1}}{x_{n+1}}\,y\otimes x +\left(\nabla \psi \cdot X\right)y\otimes \nabla \rho\right), \\
 \mu_2
 &=&
 \id-\dfrac{1}{\nabla \psi \cdot Y} y\otimes  \left(\widehat \nabla \psi -\dfrac{\psi^{n+1}}{y_{n+1}}y\right),\\
  \mu_0
 &=&
 \mu^{-1}_2\mu_1=
 \rho \id+x\otimes \na \rho+y\otimes \left[\rho x-\rho y-(X\cdot Y)\na \rho\right].
  \end{eqnarray*}
\end{proposition}

The proof of Proposition \ref{prop-Dz} will follow from Lemmas \ref{lem-nabla-t} and \ref{lem-mu-0} stated below.
Observe that for $1\leq i\leq n$ we have,
$z_i(x)=\rho(x)x_i+t(x) y_i(x)$. Differentiation in $x_j$, $1\leq j\leq n$ yields
\begin{equation}\label{eq:p zi}
\dfrac{\p z_i}{\p x_j}= \rho(x)\delta_{ij}+\dfrac{\p \rho}{\p x_j}x_i+\dfrac{\p t}{\p x_j}y_i(x)+t(x)\dfrac{\p y_i}{\p x_j},\qquad \text{$1\leq j\leq n.$}
\end{equation}
Then
\begin{equation}\label{Dz-form-00}
Dz=\rho\,\id+x\otimes \nabla \rho+y\otimes \nabla t+t\,Dy.
\end{equation}

We first find a formula for $\nabla t$.
\begin{lemma}\label{lem-nabla-t}
If $t$ satisfies $\psi(X\rho+tY)=0$, then 
\begin{equation*}\label{nabla-t}
\nabla t =-\dfrac{1}{\nabla \psi\cdot Y}\left(\rho\widehat \nabla \psi-\dfrac{\rho\psi^{n+1}}{x_{n+1}} x +\left(\nabla \psi \cdot X\right)\na \rho+
t (Dy)^t 
\left(
\widehat \nabla \psi -\dfrac{\psi^{n+1}}{y_{n+1}}y \right)
\right).
\end{equation*}
\end{lemma}
\begin{proof}

Observe that  $z_{n+1}=\rho(x)x_{n+1}+t(x)y_{n+1}=\rho(x)\sqrt{1-|x|^2}+t(x)\sqrt{1-|y|^2}$. Then for $1\leq j\leq n$ we have 
\begin{align}\label{eq:p Zn+1}
\dfrac{\p z_{n+1}}{\p x_j}=x_{n+1}\dfrac{\partial \rho}{\partial x_j}-\rho(x)\dfrac{x_j}{x_{n+1}}+y_{n+1}\dfrac{\p t}{\p x_j}-t(x)\dfrac{\sum_{k=1}^n y_k\dfrac{\partial y_k}{\partial x_j}}{y_{n+1}}.
\end{align}

Recall that the surface $\Sigma$ is defined implicitly by the equation $\psi(Z)=0,Z\in \Sigma$.
 Since on $\Sigma$ we have that $Z=X\rho+tY$ then after differentiating  $\psi(X\rho+tY)=0$ with respect to $x_j$, and $1\leq j\leq n$, and using \eqref{eq:p zi}, \eqref{eq:p Zn+1} we obtain 
\begin{eqnarray*}
0
&=&
\sum_{k=1}^{n} \psi^k\dfrac{\p z_k}{\p x_j}+\psi^{n+1}\dfrac{\p z_{n+1}}{\p x_j}\\
&=&
\sum_{k=1}^n \psi^k\left(\rho(x)\delta_{kj}+\dfrac{\p \rho}{\p x_j}x_k+\dfrac{\p t}{\p x_j}y_k+t(x)\dfrac{\partial y_k}{\p x_j}\right)\\
&&\hspace{2cm}+\psi^{n+1}\left[x_{n+1}\dfrac{\partial \rho}{\partial x_j}-\rho(x)\dfrac{x_j}{x_{n+1}}+y_{n+1}\dfrac{\p t}{\p x_j}- \dfrac{t(x)}{y_{n+1}} \sum_{k=1}^n y_k\dfrac{\partial y_k}{\partial x_j}\right]\\
&=&\rho\,\psi^j-\rho \dfrac{\psi^{n+1}}{x_{n+1}}x_j+\dfrac{\p \rho}{\p x_j}\left(\nabla\psi \cdot X\right)+\dfrac{\p t}{\p x_j}\left(\nabla \psi \cdot Y\right)+t(x)\sum_{k=1}^n\dfrac{\p y_k}{\p x_j}\left(\psi^k-\dfrac{y_k\psi^{n+1}}{y_{n+1}}\right).
\end{eqnarray*}
Hence, writing the above system in vector form yields the identity
$$
0=\rho\widehat \nabla \psi-\dfrac{\rho\psi^{n+1}}{x_{n+1}}x+\left(\nabla \psi \cdot X\right)\nabla \rho+\left(\nabla \psi \cdot Y\right)\nabla t +t \, (Dy)^t \left(\widehat \nabla \psi -\dfrac{y}{y_{n+1}} \psi^{n+1}\right).
$$
%
From Condition $\H3$, $\nabla \psi\cdot Y\neq 0$, so solving the above equation in $\na t$, we obtain Lemma \ref{nabla-t}.
\end{proof}

Given two  vectors $\xi,\eta$ in $\R^n$ and $A$ an $n\times n$ matrix, we have 
\begin{equation}\label{matrix-vector}
\xi\otimes A\eta=\xi (A \eta)^t=\xi\eta^tA^t=(\xi\otimes \eta)A^t,
\end{equation}
hence from  Lemma \ref{lem-nabla-t} we deduce that
\begin{equation}
y\otimes \nabla t=-\dfrac{1}{\nabla \psi\cdot Y}
\left[
\rho\, \left(y\otimes \widehat \nabla \psi\right)-\dfrac{\rho\,\psi^{n+1}}{x_{n+1}}\,y\otimes x +\left(\nabla \psi \cdot X\right)y\otimes \nabla \rho+
t \, y\otimes 
\left(
\widehat \nabla \psi -\dfrac{y}{y_{n+1}}  \psi^{n+1}
\right)
Dy
\right].
\end{equation}
Consequently, by virtue of \eqref{Dz-form-00} we obtain 
\begin{align*}
Dz
&=
 \rho\, \id +x\otimes \nabla \rho -\dfrac{1}{\nabla \psi \cdot Y}\left[\rho\, \left(y\otimes \widehat \nabla \psi\right)-\dfrac{\rho\psi^{n+1}}{x_{n+1}}\,y\otimes x +\left(\nabla \psi \cdot X\right)y\otimes \nabla \rho\right]\\\nonumber
&\qquad
+t \left(\id-\dfrac{1}{\nabla \psi \cdot Y} y\otimes  \left(\widehat \nabla \psi -\dfrac{\psi^{n+1}}{y_{n+1}}y\right)\right)Dy\\
&:=
\mu_1+t\mu_2 Dy.
\end{align*}

Next we recall the Morrison-Sherman formula. 
If $ \mu=\mathcal D+\xi\otimes\eta$ with $A$ invertible matrix then
\begin{equation}\label{Sherman Det}
\det \mu= \left(1+\xi\cdot \mathcal D^{-1}\eta\right)\det \mathcal D.
\end{equation}
If $1+\xi\cdot  \mathcal D^{-1} \eta\neq 0$ then 
\begin{equation}\label{Sherman}
\mu^{-1}=\mathcal D^{-1}-\frac{\mathcal D^{-1}\xi\otimes\eta \mathcal D^{-1}}{1+\xi\cdot \mathcal D^{-1}\eta}, 
\end{equation}
We show that $\mu_2=\id-\dfrac{1}{\nabla \psi \cdot Y} y\otimes  \left(\widehat \nabla \psi -
\dfrac{ \psi^{n+1}}{y_{n+1}}y\right)$ is invertible and calculate $\mu_2^{-1}$.
In fact, applying \eqref{Sherman Det} 
\begin{align}
\det \mu_2
=
1-\dfrac{1}{\nabla \psi \cdot Y} y\cdot  \left(\widehat \nabla \psi -\dfrac{ \psi^{n+1}}{y_{n+1}}y\right)
&=\dfrac{1}{\nabla \psi \cdot Y}\left(\nabla \psi \cdot Y-y\cdot\widehat \nabla \psi+\dfrac{ \psi^{n+1}}{y_{n+1}}|y|^2\right)\label{eq:Det mu2}\\
&=\dfrac{\psi^{n+1}}{\nabla \psi \cdot Y}\left(y_{n+1}+\dfrac{|y|^2}{y_{n+1}}\right)\notag\\
&=\dfrac{\psi^{n+1}}{y_{n+1}\left(\nabla \psi \cdot Y\right)}.\notag
\end{align}
%
Then, from $\H2$, $\mu_2$ is invertible and the Sherman-Woodburry-Morrison formula \eqref{Sherman} yields
\begin{align*}
\mu_2^{-1}
&=
\id+\dfrac{\dfrac{1}{\nabla \psi \cdot Y} y\otimes  \left(\widehat \nabla \psi -
\dfrac{ \psi^{n+1}}{y_{n+1}}y\right)}{\det \mu_2}\\
&=
\id+\dfrac{y_{n+1}}{\psi^{n+1}}y\otimes  \left(\widehat \nabla \psi -
\dfrac{ \psi^{n+1}}{y_{n+1}}y\right)\\
&=
\id+y\otimes \left(\dfrac{y_{n+1}}{\psi^{n+1}} \widehat \nabla \psi - y\right).
\end{align*}

We can rewrite $Dz$ as follows
$Dz=\mu_2\left(\mu_2^{-1}\mu_1+t Dy\right)$. The proof of Proposition \ref{prop-Dz} then follows from the following Lemma.

\begin{lemma}\label{lem-mu-0}
Let $\mu_0=\mu_2^{-1}\mu_1$ and $\mu_1, \mu_2$ are as in Proposition \ref{prop-Dz}.
Then 
\[
\mu_0=\rho \id+x\otimes \na \rho+y\otimes \left[\rho x-\rho y-\na \rho(X\cdot Y)\right].
\]
\end{lemma}
\begin{proof}
We have 
$$\mu_0=\mu_2^{-1}\mu_1=
\left(\id+y\otimes\left( \frac{y_{n+1}}{\psi^{n+1}}\widehat \nabla \psi -y\right)\right)
\left\{
\rho \id+x\otimes \na \rho
-\frac1{\na\psi\cdot Y}
y\otimes
\left[
\rho \widehat \nabla \psi-\rho\frac{\psi^{n+1}}{y_{n+1}} x+(\na \psi\cdot X)\na \rho
\right]
\right\}.
$$

In order to simplify $\mu_0$ we notice that for every $ a, b, c, d$ column vectors in $\R^n$,
\begin{equation*}
(a\otimes b)(c\otimes d)= ab^tcd^t=a (b\cdot c) d^t=(b\cdot c)a\otimes d.
\end{equation*}
Then 

\begin{eqnarray*}
\mu_0
&=&
\rho \id+x\otimes \na \rho-\frac1{\na\psi\cdot Y} y\otimes
\left[
\rho \widehat \nabla \psi-\rho\frac{\psi^{n+1}}{y_{n+1}}x +(\na \psi\cdot X)\na \rho
\right]
\\
&&+
\rho y\otimes \left( \frac{y_{n+1}}{\psi^{n+1}}\widehat \nabla \psi -
y\right) + \left( \frac{y_{n+1}}{\psi^{n+1}}\widehat \nabla \psi\cdot x -
y\cdot x\right)y\otimes \na \rho \\
&&
-\frac1{\na \psi\cdot  Y}\left( \frac{y_{n+1}}{\psi^{n+1}}\widehat \nabla \psi\cdot y -
|y|^2\right) y\otimes \left[
\rho \widehat \nabla \psi-\rho\frac{\psi^{n+1}}{y_{n+1}} x+(\na \psi\cdot X)\na \rho
\right].
\\
\end{eqnarray*}
Next we regroup the terms by separating those containing $y$ in tensorial products
\begin{eqnarray*}
\mu_0
&=&
\rho \id+x\otimes \na \rho\\
&&+
y\otimes
\Bigg\{
-\frac1{\na \psi\cdot  Y}\left[
\rho \widehat \nabla \psi-\rho\frac{\psi^{n+1}}{y_{n+1}}x +(\na \psi\cdot X)\na \rho
\right] 
+
\rho\left( \frac{y_{n+1}}{\psi^{n+1}}\widehat \nabla \psi -
y\right)
+
 \left( \frac{y_{n+1}}{\psi^{n+1}}\widehat \nabla \psi \cdot x-
y\cdot x\right)\na \rho\\
&&\hspace{2cm}
-\frac1{\na\psi\cdot Y} \left( \frac{y_{n+1}}{\psi^{n+1}}\widehat \nabla \psi \cdot y-
|y|^2\right) \left[
\rho \widehat \nabla \psi-\rho\frac{\psi^{n+1}}{y_{n+1}}x +(\na \psi\cdot X)\na \rho
\right]
\Bigg\}.
\end{eqnarray*}

The expression in the brackets can be simplified as follows: 
\begin{eqnarray*}
\Bigg\{\dots\Bigg\} 
&=&
-\frac1{\na\psi\cdot Y} \left( \frac{y_{n+1}}{\psi^{n+1}}\widehat \nabla \psi \cdot y+1-
|y|^2\right)
\left[
\rho \widehat \nabla \psi-\rho\frac{\psi^{n+1}}{y_{n+1}} x+(\na \psi\cdot X)\na \rho
\right]\\
&& +
\rho\left( \frac{y_{n+1}}{\psi^{n+1}}\widehat \nabla \psi -
y\right)
+
 \left( \frac{y_{n+1}}{\psi^{n+1}}\widehat \nabla \psi \cdot x-
y\cdot x\right)\na \rho\\
&=&
-\frac1{\na\psi\cdot Y} \left( \frac{y_{n+1}}{\psi^{n+1}}\widehat \nabla \psi \cdot y+
y_{n+1}^2\right)\left[
\rho \widehat \nabla \psi-\rho\frac{\psi^{n+1}}{y_{n+1}}x +(\na \psi\cdot X)\na \rho
\right]\\
&& +
\rho\left( \frac{y_{n+1}}{\psi^{n+1}}\widehat \nabla \psi -
y\right)
+
\left( \frac{y_{n+1}}{\psi^{n+1}}\widehat \nabla \psi \cdot x-
y\cdot x\right)\na \rho \\
&=&
-\frac{y_{n+1}}{\psi^{n+1}}
\left[
\rho \widehat \nabla \psi-\rho\frac{\psi^{n+1}}{y_{n+1}}x +(\na \psi\cdot X)\na \rho
\right]\\
&& +
\rho\left( \frac{y_{n+1}}{\psi^{n+1}}\widehat \nabla \psi -
y\right)
+
\na \rho \left( \frac{y_{n+1}}{\psi^{n+1}}\widehat \nabla \psi \cdot x-
y\cdot x\right)\\
&=&
\rho x-\rho y-\na \rho(X\cdot Y).
\end{eqnarray*}
Then 
$
\mu_0=\rho \id+x\otimes \na \rho+y\otimes \left[\rho x-\rho y-(X\cdot Y)\na \rho\right],
$
as required.
\end{proof}

\begin{remark}
From Proposition \ref{prop-Dz} it follows that $Dz=\mu_2(\mu_0+tD^2 y)$, where 
the matrix $\mu_0$ does not depend on the Hessian of $\rho$ thanks to Lemma \ref{lem-mu-0}.
Therefore, we expect that $Dy$ will contain $D^2\rho$.   
\end{remark}

\subsection{Formula for $Dy$}
In this section, we show the following result. We refer the reader to the list of notations in the introduction.
\begin{proposition}\label{prop:formula for Dy}
If the ray with direction $X$ is refracted by $\nu$ into the unit direction $Y=(y,y_{n+1})$ according to the Snell's law \eqref{eq:Snell} in $\R^{n+1}$, then
\begin{eqnarray*}
Dy
&=&
\left[\kappa-(\rho+(x\cdot \nabla \rho)b)\right]\id+\left[\alpha(\nabla\rho-( \rho+x\cdot\nabla \rho)x)-2b\,x\right]\otimes \nabla \rho+\\
&&+
\Big\{b(\id-x\otimes x)+\beta\left[\nabla \rho-(\rho+x\cdot\nabla \rho)x\right]\otimes\left[\nabla \rho-(x\cdot\nabla \rho)x\right]\Big\}D^2\rho.
\end{eqnarray*}

\end{proposition}

We first derive a formula for $Y=(y,y_{n+1})$.
\begin{lemma}\label{lm:formula for Y}
The ray with direction $X$ is refracted by $\Sigma$ into the direction $Y$ such that
\begin{equation*}
Y=\kappa\, X+ b\left((\nabla \rho,0)-X(\rho+x\cdot \nabla\rho)\right).
\end{equation*}
\end{lemma}

\begin{proof}
From Snell's law \eqref{eq:Snell}, we have
\begin{equation*}
Y=\frac{n_1}{n_2} X+\left(-\frac{n_1}{n_2}(\nu\cdot X)+\sqrt{1-\left(\frac{n_1}{n_2}\right)^2(1-(\nu\cdot X)^2)}\right)\nu=(y,y_{n+1}).
\end{equation*}
From \eqref{eq:unit normal}
$X\cdot \nu=\frac\rho{\sqrt a}$, hence
\begin{align*}
Y&=\kappa\, X-\left(\kappa(\nu\cdot X)-\sqrt{1-\kappa^2(1-(\nu\cdot X)^2)}\right)\nu\\
&=\kappa\, X+\left(\kappa \dfrac{\rho}{\sqrt{a}}-\sqrt{1-\kappa^2\left(1-\left(\dfrac{\rho}{\sqrt{a}}\right)^2\right)}\right)\dfrac{(\nabla \rho,0)-X(\rho+x\cdot \nabla \rho)}{\sqrt{a}}\\
&=\kappa\, X+\dfrac{1}{a}\left(\kappa \rho-\sqrt{a-\kappa^2(a-\rho^2)}\right)\left[(\nabla \rho,0)-(\rho+x\cdot \nabla\rho)X\right].
\end{align*}

\end{proof}

We next calculate $\nabla b$.
\begin{lemma}\label{nabla-b}
Using the notations of Proposition \ref{prop:formula for Dy}
$$\nabla b= \alpha \nabla \rho+\beta\, D^2\rho\, (\nabla \rho-(x\cdot \nabla \rho)x).$$
\end{lemma}

\begin{proof}
The proof follows from a straightforward computation. We have
\begin{align*}
\dfrac{\p a}{\p x_j}
&=
2\left[\rho \dfrac{\p \rho}{\p x_j}+\sum_{k=1}^n\dfrac{\p^2\rho}{\p x_j\p x_k}\dfrac{\p \rho}{\p x_k}-(x\cdot \nabla \rho)\left(\dfrac{\p\rho}{\p x_j}+\sum_{k=1}^n \dfrac{\p^2 \rho}{\p x_j\p x_k}x_k\right)\right]\\
&=
2(\rho-x\cdot\nabla \rho)\dfrac{\p \rho}{\p x_j}+2\sum_{k=1}^n\dfrac{\p^2\rho}{\p x_j\p x_k}\left(\dfrac{\p \rho}{\p x_k}-(x\cdot \nabla \rho)x_k\right).
\end{align*}
Then 
\begin{equation}\label{nabla-a}
\nabla a=2(\rho-x\cdot\nabla \rho)\nabla \rho+2D^2\rho\,(\nabla \rho-(x\cdot\nabla \rho)x).
\end{equation}

From the formula of $b$ and $q$, notice that
\begin{align}\label{def-b}
b&=\dfrac{1}{a}\left(\kappa \rho-\sqrt{a-\kappa^2(a-\rho^2)}\right)
=
\frac1a \frac{\kappa^2\rho^2-\kappa^2\rho^2-a(1-\kappa^2)}{\kappa \rho+\sqrt{\kappa^2\rho^2+a(1-\kappa^2)}}\\
&=
\frac{\kappa^2-1}{\kappa \rho+\sqrt{\kappa^2\rho^2+a(1-\kappa^2)}}\\
&=
\frac{\kappa^2-1}{\kappa\rho+q}.\nonumber
\end{align}
Applying the logarithmic derivative yields
\begin{eqnarray*}
\frac{\na b}b
&=&
-\frac{\kappa\na \rho+\na q}{\kappa\rho+q}\\
&=&
-\frac1{\kappa\rho+q}\left[\kappa\na \rho +\frac{2\kappa^2\rho\na \rho+(1-\kappa^2)\na a}{2q}\right].
\end{eqnarray*}
Substituting the formula for $\na a$ from \eqref{nabla-a}, we get 

\begin{eqnarray*}
\na b
&=&
 -\frac b{\kappa\rho+q} \left[\kappa\na \rho +\frac{2\kappa^2\rho\na \rho+(1-\kappa^2)[2(\rho-x\cdot\nabla \rho)\nabla \rho
+ 2D^2\rho\,(\nabla \rho-(x\cdot\nabla \rho)x)]}{2q}\right]\\
&=&
- \frac b{\kappa\rho+q} \left[\left(\kappa+\frac1{q}[\kappa^2\rho+(1-\kappa^2)(\rho-x\cdot \na \rho)]\right)\na \rho+\frac{1-\kappa^2}q D^2\rho(\na\rho-(x\cdot\na\rho)x)\right]\\
&=&
- \frac b{\kappa\rho+q} \left[\kappa+\frac1{q}[\kappa^2\rho+(1-\kappa^2)(\rho-x\cdot \na \rho)]\right]\na \rho
+\frac{b(\kappa^2-1)}{q(\kappa\rho+q)}D^2\rho\left(\na\rho-(x\cdot\na\rho)x\right)\\
&=&
-\dfrac{b^2}{\kappa^2-1}\left[\dfrac{\kappa q+\rho+(\kappa^2-1)x\cdot\nabla \rho}{q}\right]\na\rho+\dfrac{b^2}{q} D^2\rho\left[\na\rho-(x\cdot\na\rho)x\right]
\end{eqnarray*}
and the proof of the Lemma  follows.
\end{proof}

We are now ready to derive the formula for $Dy$ in Proposition \ref{prop:formula for Dy}.
\begin{proof}[{\bf Proof of Proposition \ref{prop:formula for Dy}}]
Writing $y=(y_1,\cdots,y_n)$, we have from Lemma \ref{lm:formula for Y}
$$y_i=\kappa \,x_i+b\left(\dfrac{\p \rho}{\p x_i}-x_i(\rho+x\cdot \nabla \rho)\right)=\kappa \, x_i+b\left(\dfrac{\p \rho}{\p x_i}-x_i\left(\rho+\sum_{k=1}^n \dfrac{\p \rho}{\p x_k}x_k\right)\right), \quad i=1, \dots, n.
$$
Differentiating in $x_j$ we get 
\begin{align*}
\dfrac{\p y_i}{\p x_j}&= \kappa \,\delta_{ij}+b\left[\dfrac{\p^2 \rho}{\p x_j\p x_i}-\delta_{ij}\left(\rho+\sum_{k=1}^n \dfrac{\p \rho}{\p x_k}x_k\right)-x_i\left(\dfrac{\p \rho }{\p x_j}+\sum_{k=1}^n \dfrac{\p\rho}{\p x_k}\delta_{jk}+\sum_{k=1}^n\dfrac{\p^2\rho}{\p x_j\p x_k}x_k\right)\right]\\
&\qquad\,\,\,+\dfrac{\p b}{\p x_j}\left(\dfrac{\p \rho}{\p x_i}-x_i\left(\rho+\sum_{k=1}^n \dfrac{\p \rho}{\p x_k}x_k\right)\right)\\
&=\kappa \,\delta_{ij}+b\left[\dfrac{\p^2 \rho}{\p x_j\p x_i}-\delta_{ij}\left(\rho+x\cdot \nabla \rho\right)-x_i\left(2\dfrac{\p \rho}{\p x_j}+\sum_{k=1}^n\dfrac{\p^2\rho}{\p x_j\p x_k}x_k\right)\right] +\dfrac{\p b}{\p x_j}\left(\dfrac{\p \rho}{\p x_i}-x_i\left(\rho+x\cdot \nabla \rho \right)\right).
\end{align*}
Consequently 
\begin{align*}
Dy&=\kappa\, \id+b \left(D^2\rho-\id(\rho+x\cdot \nabla\rho)-x\otimes\left(2 \nabla \rho+(D^2\rho) x\right)\right)+\left(\nabla \rho-x(\rho+x\cdot \nabla \rho)\right)\otimes \nabla b.
\end{align*}

Using property \eqref{matrix-vector} of tensor product, 
\begin{eqnarray*}
Dy
&=&
\kappa \, \id -(\rho+x\cdot \nabla \rho)(b\, \id +x\otimes \nabla b)+b(\id-x\otimes x)D^2\rho+\nabla \rho \otimes \nabla b-2b\, x\otimes \nabla \rho.
\end{eqnarray*}
In view of Lemma \ref{nabla-b}, we isolate the Hessian of $\rho$ to get
\begin{eqnarray*}
Dy
&=&
\kappa \, \id-(\rho+x\cdot \nabla \rho)\left(b\,\id+x\otimes \left[\alpha \nabla \rho +\beta D^2\rho(\nabla\rho-(x\cdot\nabla \rho)x\right)\right]+\\
&&+
b(\id-x\otimes x)D^2\rho+\nabla \rho\otimes \left[\alpha \nabla \rho +\beta D^2\rho(\nabla\rho-(x\cdot\nabla \rho)x)\right] -2b \,x\otimes \nabla \rho.
\end{eqnarray*}
Regrouping the terms, and using  \eqref{matrix-vector} again we get 
\begin{eqnarray*}
Dy
&=&
\left[\kappa-(\rho+(x\cdot \nabla \rho)b)\right]\id+\left[\alpha(\nabla\rho-( \rho+x\cdot\nabla \rho)x)-2b\,x\right]\otimes \nabla \rho+\\
&&+
\Big\{b(\id-x\otimes x)+\beta\left[\nabla \rho-(\rho+x\cdot\nabla \rho)x\right]\otimes\left[\nabla \rho-(x\cdot\nabla \rho)x\right]\Big\}D^2\rho.
\end{eqnarray*}

\end{proof}

\subsection{Proof of Theorem \ref{thm-1}}

Plugging the formula of $Dy$ in \eqref{eq:Dz form 1} we get 

\begin{eqnarray*}
Dz&=&\mu_2\Bigg[\mu_0+
t\bigg\{
\left[\kappa-(\rho+(x\cdot \nabla \rho)b)\right]\id+\left[\alpha(\nabla\rho-( \rho+x\cdot\nabla \rho)x)-2b\,x\right]\otimes \nabla \rho\bigg\}\\
&&\hspace{2cm}+
t\big\{b(\id-x\otimes x)+\beta\left[\nabla \rho-(\rho+x\cdot\nabla \rho)x\right]\otimes\left[\nabla \rho-(x\cdot\nabla \rho)x\right]\big\} D^2\rho\nonumber
\Bigg].
\end{eqnarray*}
Applying  \eqref{eq:Det mu2}
\begin{equation}\label{eq:simplified Dz}
\det\, Dz
=
 \dfrac{1}{\nabla \psi \cdot Y}  \dfrac{{ \psi^{n+1}}}{y_{n+1}}
\det \Bigg[
\mu_0+
t\bigg\{
(\kappa-b(\rho+x\cdot \na \rho))\id+\big\{\alpha(\na \rho-x(\rho+x\cdot \na \rho))-2bx\big\}\otimes \na \rho\bigg\}
+t \mathcal MD^2\rho\Bigg],
\end{equation}
with
 \begin{equation}\label{eq:Formula for M}
 \mathcal M= b(\id-x\otimes x)+\beta\left[\nabla \rho-(\rho+x\cdot\nabla \rho)x\right]\otimes\left[\nabla \rho-(x\cdot\nabla \rho)x\right].
 \end{equation}
It remains to show that $\mathcal M$ is invertible, compute its inverse of $\mathcal M$ and multiply the Hessian free matrix by it.

First we prove 
\begin{lemma}\label{lem-M-inverse}
We have  
\[
\det\, \mathcal M= \left(1+\frac bq[|\na\rho|^2-(\na\rho\cdot x)(\rho+\na\rho\cdot x)]\right)b^nx_{n+1}.
\]
Moreover, in a neighborhood of $x=0$, $\mathcal M$ is invertible and
\[
\mathcal M^{-1}
=
\dfrac{1}{b}\left\{Id+\dfrac{x\otimes x}{x_{n+1}^2}-Q\left(\na \rho -\frac{\rho}{x_{n+1}}x\right)\otimes\nabla \rho\right\}.
\]
\end{lemma}

\begin{proof}
Let $\A=b(\id -x\otimes x)$ then from \eqref{Sherman} $\det \,\A=b^n(x_{n+1}^2)=b^nx_{n+1}\neq 0$ since $x_{n+1}>0$ and $\kappa\neq 1$ then $\A$ is invertible and 
$$\A^{-1}=\dfrac{1}{b}\left(\id +\dfrac{x\otimes x}{x_{n+1}^2}\right).$$

Denote $\xi=\beta(\na \rho -(\rho+\na \rho\cdot x)x)$ and $\eta=\na \rho -(\na \rho\cdot x)x$. We calculate $1+\eta\cdot \A^{-1}\xi$. 
Note that $\A$ is symmetric. Thus  
\begin{eqnarray*}
\eta\cdot\A^{-1}\xi
&=&
\xi\cdot\A^{-1}\eta=
\frac1b\left\{\xi\cdot \eta+\frac1{x_{n+1}^2}(\xi\cdot x)(x\cdot  \eta)\right\}.
\end{eqnarray*}
Substituting $\eta, \xi$ into the last formula and simplifying the 
resulted expression we obtain
\begin{eqnarray*}
1+\eta\cdot\A^{-1}\xi
&=&
1+\frac \beta b\left[ (\na \rho -x(\rho+\na \rho\cdot x))\cdot (\na \rho -x(\na \rho\cdot x))+\frac1{x_{n+1}^2}
(\na \rho\cdot x -|x|^2(\rho+\na \rho\cdot x)) (\na \rho\cdot x -|x|^2(\na \rho\cdot x))\right]\\
&=&
1+\frac b q\bigg[ (\na \rho -x(\rho+\na \rho\cdot x))\cdot (\na \rho -x(\na \rho\cdot x))+
(\na \rho\cdot x -|x|^2(\rho+\na \rho\cdot x)) (\na \rho\cdot x )\bigg]\\
&=&
1+\frac b q\bigg[ |\na\rho |^2-(\na \rho\cdot x)^2 -\na \rho \cdot x(\rho+\na \rho\cdot x)+
|x|^2\na \rho \cdot x (\rho+\na \rho\cdot x)+
(\na \rho\cdot x -|x|^2(\rho+\na \rho\cdot x)) (\na \rho\cdot x )
\bigg]\\
&=&
1+\frac b q\bigg[ |\na\rho |^2 -\na \rho \cdot x(\rho+\na \rho\cdot x)
\bigg].\\
\end{eqnarray*}
 Combining these computations we get 
\begin{eqnarray*}
\det \, \mathcal M
&=&\det \A \,\det (\id+\eta\otimes \A^{-1}\xi)=
\left(1+\frac bq[|\na\rho|^2-(\na\rho\cdot x)(\rho+\na\rho\cdot x)]\right)b^nx_{n+1}
.
\end{eqnarray*}
Notice that if $\kappa>1$ then $b>0$ and clearly $\det \mathcal M> 0$. 

If $\kappa<1$ which is the case we will study later for regularity, we have at $x=0$ using \eqref{def-b}
\begin{align*}
1+\eta\cdot\A^{-1}\xi &=1+\frac bq|\nabla \rho|^2\\
&=1+\dfrac{\kappa^2-1}{\kappa \rho q+q^2}|\nabla\rho|^2\\
&=1+\dfrac{\kappa^2-1}{\kappa \rho q+(1-\kappa^2)a+\kappa^2 \rho^2}|\nabla \rho|^2\\
&=1+\dfrac{\kappa^2-1}{\kappa\rho q+(1-\kappa^2)(\rho^2+|\nabla \rho|^2)}|\nabla \rho|^2.
\end{align*}
Since $\kappa<1$ then the denominator is larger than $(1-\kappa^2)|\nabla \rho|^2$, and hence 
$\dfrac{b}{q}|\nabla \rho|^2<-1$ and so $1+\eta\cdot\A^{-1}\xi<0$, and $\mathcal M$ is invertible at $x=0$.
Therefore by continuity $\mathcal M$ is invertible in a neighborhood of $x=0$. 

We next calculate $\mathcal M^{-1}$. 
We have
\begin{eqnarray*}
\A^{-1}\xi\otimes\eta \A^{-1}
&=&
\frac1b \left[\id +\frac{x\otimes x}{x_{n+1}^2}\right]\beta(\na \rho -(\rho+\na \rho\cdot x)x)\otimes
(\na \rho -(\na \rho\cdot x)x)\frac1b \left[\id +\frac{x\otimes x}{x_{n+1}^2}\right]\\
&=&
\frac \beta{b^2} \left[\id +\frac{x\otimes x}{x_{n+1}^2}\right] (\na \rho -(\rho+\na \rho\cdot x)x)\otimes
(\na \rho -x(\na \rho\cdot x)) \left[\id +\frac{x\otimes x}{x_{n+1}^2}\right].
\end{eqnarray*}

We first compute the vector $\left[\id +\frac{x\otimes x}{x_{n+1}^2}\right]\xi$
and then take its tensor product with $\eta \left[\id +\frac{x\otimes x}{x_{n+1}^2}\right]$
and then simplify the resulted expression in order to get  
\begin{eqnarray*}
\A^{-1}\xi\otimes\eta \A^{-1}
&=&
\frac \beta{b^2} 
\left[\na \rho -(\rho+\na \rho\cdot x)x +\frac{\na \rho\cdot x-|x|^2(\rho+\na \rho\cdot x)}{x_{n+1}}x\right] 
\otimes
(\na \rho -(\na \rho\cdot x)x) \left[\id +\frac{x\otimes x}{x_{n+1}^2}\right]\\
&=&
\frac 1{q} 
\left[\na \rho -\rho\, x -\frac{|x|^2\rho}{x_{n+1}^2}x\right] 
\otimes
 \left[\na \rho -(\na \rho\cdot x)x +\frac{\na\rho\cdot x-|x|^2(\na \rho\cdot x)}{x_{n+1}^2}x\right]\\
 &=&
\frac 1 {q} \left(\na \rho -\frac{\rho}{x_{n+1}^2}x\right)
\otimes
 \na \rho.
\end{eqnarray*}

Combining our obtained calculation we get
\begin{align*}
\mathcal M^{-1}&=\dfrac{1}{b}\left(\id +\dfrac{x\otimes x}{x_{n+1}^2}\right)-\dfrac{\frac 1 {q} }{1+\dfrac{b}{q}\left(|\nabla q|^2-\nabla\rho\cdot x(\rho+\nabla \rho\cdot x)\right)}\left(\na \rho -\frac{\rho}{x_{n+1}^2}x\right)\otimes\nabla \rho\\
&=\dfrac{1}{b}\left\{\id +\dfrac{x\otimes x}{x_{n+1}^2}-\dfrac{b}{q+b\left(|\nabla q|^2-\nabla\rho\cdot x(\rho+\nabla \rho\cdot x)\right)}\left(\na \rho -\frac{\rho}{x_{n+1}^2}x\right)\otimes\nabla \rho\right\}.
\end{align*}
\end{proof}

\begin{remark}\label{formula second for Q}\rm
Notice that from \eqref{def-b}, and the formula for $Q$ we have
\begin{align*}
Q&=\dfrac{\dfrac{\kappa^2-1}{\kappa \rho+q}}{q+\dfrac{\kappa^2-1}{\kappa \rho+q}\left[|\nabla \rho|^2-(\nabla \rho\cdot x)(\rho+(\nabla \rho\cdot x))\right]}\\
&=\dfrac{\kappa^2-1}{\kappa \rho q+q^2+(\kappa^2-1)\left[|\nabla \rho|^2-(\nabla \rho\cdot x)(\rho+(\nabla \rho\cdot x))\right]}\\
&=\dfrac{\kappa^2-1}{\kappa \rho q+(1-\kappa^2)[\rho^2+|\nabla \rho|^2-(x\cdot \nabla \rho)^2]+(\kappa^2-1)\left[|\nabla \rho|^2-(\nabla \rho\cdot x)(\rho+(\nabla \rho\cdot x))\right]}\\
&=\dfrac{1}{\rho}\dfrac{\kappa^2-1}{\kappa q +\rho-(\kappa^2-1)\nabla \rho \cdot x}.
\end{align*}
\end{remark}

\begin{proof}[{\bf Proof of Theorem \ref{thm-1}}]
We want to isolate the Hessian  $D^2\rho$. We do this in a neighborhood of $x=0$. From \eqref{eq:simplified Dz}, we write
\begin{align*}
\det\, Dz=\dfrac{t^n}{\nabla \psi \cdot Y}  \dfrac{{ \psi^{n+1}}}{y_{n+1}} \det\, \mathcal M\,
\det \Bigg[
\frac{1}{t}\mathcal M^{-1} \mu_0+
\, \mathcal M^{-1}\bigg\{
\ (\kappa-b(\rho+x\cdot \na \rho))Id+\big\{\alpha(\na \rho-x(\rho+x\cdot \na \rho))-2bx\big\}\otimes \na \rho\bigg\}
+D^2\rho\Bigg].
\end{align*} 
Letting 
\begin{align}
\mathcal A&=b\,\mathcal M^{-1} \mu_0 \label{eq:mathcal A},\\
\mathcal B&= b\,\mathcal M^{-1}\bigg\{
\ (\kappa-b(\rho+x\cdot \na \rho))Id+\big\{\alpha(\na \rho-x(\rho+x\cdot \na \rho))-2bx\big\}\otimes \na \rho\bigg\},\label{eq:mathcal B}
\end{align}
and replacing in \eqref{compact}, we conclude the proof of Theorem  \ref{thm-1}.
\end{proof}

\section{Simplifying the matrix coefficients}\label{sec:Simplifying}
In this section we simplify the form of the matrices $\mathcal A$ and $\mathcal B$ given in \eqref{eq:mathcal A}, \eqref{eq:mathcal B}.

\subsection{The matrix $\mathcal A$}
We show the following simplified formula for $\mathcal A$.
\begin{lemma}\label{lm:simplified mathcal A}
Let $\mathcal A$ be the matrix given in \eqref{eq:mathcal A}, then 
\[
\mathcal A=\rho\left(\id+\dfrac{x\otimes x}{x_{n+1}^2}\right)-\dfrac{\kappa^2-1}{\rho}\nabla\rho\otimes\nabla\rho+\rho(1-\kappa+b(\rho+x\cdot\nabla \rho))\left[\gamma \nabla \rho+\dfrac{\kappa-\rho\gamma}{x_{n+1}^2}x\right]\otimes x.
\]

\end{lemma}
\begin{proof}
From Lemmas \ref{lem-mu-0}, and  \ref{lem-M-inverse}, 
\begin{eqnarray*}
\mathcal A
&=&
b\mathcal M^{-1}\mu_0\\
&=&
\left\{\left[\id+\dfrac{x\otimes x}{x_{n+1}^2}\right]-Q\left(\nabla \rho-\dfrac{\rho}{x_{n+1}^2}x\right)\otimes \nabla \rho\right\}\left\{\rho \id+x\otimes \na \rho+y\otimes \left[\rho x-\rho y-\na \rho(X\cdot Y)\right]\right\}\\
&=&
\rho\left\{\left[\id+\dfrac{x\otimes x}{x_{n+1}^2}\right]-Q\left(\nabla \rho-\dfrac{\rho}{x_{n+1}^2}x\right)\otimes \nabla \rho\right\}\\
&&\,
+\,  x\otimes \nabla \rho+\dfrac{|x|^2}{x_{n+1}^2}x\otimes \nabla \rho-Q(x\cdot \nabla \rho)\left(\nabla \rho-\dfrac{\rho}{x_{n+1}^2}x\right)\otimes \nabla \rho\\
&&\,
+\underbrace{\left\{y+\dfrac{x\cdot y}{x_{n+1}^2}x-Q(\nabla \rho\cdot y)\left(\nabla \rho-\dfrac{\rho}{x_{n+1}^2}x\right)\right\}\otimes[\rho x-\rho y-\nabla \rho(X\cdot Y)]}_J.
\end{eqnarray*}
We simplify  $J$ and express it as a linear combination of 
$x\otimes x, x\otimes \na \rho, \na \rho \otimes x$ and $\na \rho \otimes\na \rho$. 

From Lemma \ref{lm:formula for Y}
\[%
y=\kappa x+b(\nabla \rho-(\rho+x\cdot \nabla \rho)x),\qquad\qquad\qquad
X\cdot Y=\kappa -b\rho,
\]and so
\begin{align*}
\rho x -\rho y-(X\cdot Y) \nabla \rho&=\rho x-\rho\left[\kappa x+b(\nabla \rho-(\rho+x\cdot \nabla \rho)x)\right]-(\kappa-b\rho)\na \rho\\
&=\rho(1-\kappa+b(\rho+x\cdot \nabla \rho))x-\kappa\nabla \rho.
\end{align*}
On the other hand
\begin{align*}
y+\dfrac{x\cdot y}{x_{n+1}^2}x-Q(\nabla \rho\cdot y)\left(\nabla \rho-\dfrac{\rho}{x_{n+1}^2}x\right)&=\left[\kappa -b(\rho+x\cdot \nabla \rho)\right]x+b\nabla \rho+\dfrac{[\kappa-b(\rho+x\cdot \nabla \rho)]|x|^2+bx\cdot \nabla \rho}{x_{n+1}^2}x\\
&\qquad-Q\left([\kappa -b(\rho+x\cdot \nabla \rho)]x\cdot \nabla \rho+b|\nabla \rho|^2\right)\left(\nabla \rho-\dfrac{x\rho}{x_{n+1}^2}\right)\\
&:=\alpha_1 x+\gamma y
\end{align*}
with
\begin{align*}
\alpha_1&=\kappa-b(\rho+x\cdot \nabla\rho)+\dfrac{[\kappa-b(\rho+x\cdot \nabla \rho)]|x|^2+b x\cdot \nabla \rho}{x_{n+1}^2}+\dfrac{Q([\kappa -b(\rho+x\cdot \nabla \rho)]x\cdot \nabla \rho+b|\nabla \rho|^2)\rho}{x_{n+1}^2}\\
&=\dfrac{\kappa-\rho\left\{b-Q([\kappa -b(\rho+x\cdot \nabla \rho)]x\cdot \nabla \rho+b|\nabla \rho|^2)\right\}}{x_{n+1}^2}=\dfrac{\kappa-\rho\gamma}{x_{n+1}^2}.
\end{align*}
Observe 
\begin{eqnarray}\label{imp gamma}
\gamma
&=&
b-Q\left([\kappa- b(\rho+x\cdot \nabla \rho)]x\cdot \nabla \rho+b|\nabla \rho|^2\right)\\ \nonumber
&=&
b-Q\left(\kappa (x\cdot \na \rho) +b\left[|\nabla \rho|^2-(x\cdot \nabla \rho)(\rho+x\cdot \nabla \rho)\right]\right)\\\nonumber
&=&
b-Q\kappa (x\cdot \na \rho)+Qq-b\\\nonumber
&=&
Q(q-\kappa(x\cdot \na \rho)),
\end{eqnarray}
where we used the formula for $Q$ to get the last line.
Hence 
\begin{align*}
J&=\left(\dfrac{\kappa-\rho\,\gamma}{x_{n+1}^2}x+\gamma\nabla \rho\right)\otimes \left\{\rho\left(1-\kappa+b(\rho+x\cdot \nabla \rho)\right)x-\kappa \nabla \rho\right\}\\
&=\left(\dfrac{\kappa-\rho\,\gamma}{x_{n+1}^2}\rho(1-\kappa+b(\rho+x\cdot \nabla \rho))\right)x\otimes x \\
&\, \
+\gamma\rho(1-\kappa+b(\rho+x\cdot \nabla \rho))\nabla \rho \otimes x-\kappa\dfrac{\kappa-\rho \gamma}{x_{n+1}^2}x\otimes \nabla \rho-\kappa\gamma \nabla \rho\otimes \nabla \rho.\\
\end{align*}

We write 
$$\mathcal A= \rho Id+A_1 x\otimes \nabla \rho +A_2 \nabla \rho \otimes x+ A_3 x\otimes x+A_4 \nabla \rho\otimes \nabla \rho.$$
Notice that
$$ A_1=\dfrac{Q\rho^2+1+Q\rho(x\cdot \nabla \rho)-\kappa(\kappa -\rho \gamma)}{x_{n+1}^2}.$$

We claim that $A_1=0$. In fact, Remark \ref{formula second for Q} and \eqref{imp gamma} yield \begin{eqnarray*}
(x_{n+1}^2) A_1&=&
1-\kappa^2+\rho Q\left[(\rho+x\cdot \na \rho)+\kappa(q-\kappa(x\cdot \na \rho))\right]\\
&=&
1-\kappa^2+\rho Q\left[\kappa q+\rho+(1-\kappa^2)x\cdot \na \rho\right]\\
&=&
0.
\end{eqnarray*}
We also have
\begin{align*}
A_2&=\gamma \rho(1-\kappa+b(\rho+x\cdot \nabla \rho))\\
A_3&= \dfrac{\rho}{x_{n+1}^2}+\dfrac{\kappa-\rho \gamma}{x_{n+1}^2}\rho(1-\kappa+b(\rho+x\cdot \nabla \rho)),\\
A_4&=-Q\rho-Q(x\cdot \nabla \rho)-\kappa \gamma\\
&=-\dfrac{x_{n+1}^2}{\rho}\dfrac{Q\rho^2+1+Q\rho(x\cdot \nabla \rho)-\kappa(\kappa-\rho\gamma)}{x_{n+1}^2}-\dfrac{\kappa^2-1}{\rho}\\
&=-\dfrac{x_{n+1}^2}{\rho}A_1-\dfrac{\kappa^2-1}{\rho}=-\dfrac{\kappa^2-1}{\rho}
\end{align*}
and hence the Lemma follows.
\end{proof}

\subsection{The matrix $\mathcal B$}
We simplify now the matrix $\mathcal B$, and show the following
\begin{lemma}\label{lm:simplified B}
Given $\mathcal B$ in \eqref{eq:mathcal B}, we can write
\begin{equation}\label{def-B}
\mathcal B= \sigma \left\{\id+\dfrac{x\otimes x}{x_{n+1}^2}\right\}
+
\dfrac{Q\rho F-\alpha\rho-2b}{x_{n+1}^2}x\otimes \na \rho
+
(\alpha-QF)\na \rho\otimes \na \rho.
\end{equation}
\end{lemma}

\begin{proof}
We have from \eqref{eq:mathcal B}, and Lemma \ref{lem-M-inverse}
\begin{align*}
\mathcal B&= \left\{\left[\id+\dfrac{x\otimes x}{x_{n+1}^2}\right]-Q\left(\nabla \rho-\dfrac{\rho}{x_{n+1}^2}x\right)\otimes \nabla \rho\right\}\left\{\sigma\id+\left[\alpha \nabla \rho-(\alpha(\rho+x\cdot\nabla \rho)+2b)x\right]\otimes\na \rho\right\}\\
&=\sigma\left\{ \id+\dfrac{x\otimes x}{x_{n+1}^2}-Q\left(\nabla \rho-\dfrac{\rho}{x_{n+1}^2}x\right)\otimes \nabla \rho\right\}+\left[\alpha \nabla \rho-(\alpha(\rho+x\cdot\nabla \rho)+2b)x\right]\otimes\nabla \rho\\
&\quad +\dfrac{1}{x_{n+1}^2}\left(\alpha x\cdot \nabla \rho-(\alpha(\rho+x\cdot\nabla \rho)-2b)|x|^2\right)x\otimes\nabla \rho -Q\left(\alpha |\nabla \rho|^2-(\alpha(\rho+x\cdot \nabla \rho)+2b)x\cdot \nabla \rho\right)\left(\nabla \rho-\dfrac{\rho}{x_{n+1}^2}x\right)\otimes \nabla \rho\\
&=\sigma \id+\sigma \dfrac{x\otimes x}{x_{n+1}^2} +\Lambda_0 x\otimes \na \rho +\Lambda \na \rho \otimes\na \rho.
\end{align*}
The coefficient of $x\otimes \nabla \rho$ is
{
\begin{align*}
\hspace{-1cm}\Lambda_0&:=\dfrac{\sigma Q\rho}{x_{n+1}^2}-\alpha(\rho+x\cdot \nabla \rho)-2b+\dfrac{1}{x_{n+1}^2}\left(\alpha\,x\cdot \nabla \rho-(\alpha(\rho+x\cdot\nabla \rho)+2b)|x|^2\right)\\
&\qquad+\dfrac{Q\rho}{x_{n+1}^2}\left(\alpha |\nabla \rho|^2-(\alpha(\rho+x\cdot \nabla \rho)+2b)x\cdot \nabla \rho\right)\\
&=\dfrac{Q\rho\sigma}{x_{n+1}^2}+\dfrac{1}{x_{n+1}^2}\left(\alpha x\cdot \nabla \rho-\alpha(\rho+x\cdot \nabla \rho)-2b\right)+\dfrac{Q\rho}{x_{n+1}^2}\left(\alpha |\nabla \rho|^2-(\alpha(\rho+x\cdot \nabla \rho)+2b)x\cdot \nabla \rho\right)\\
&=\dfrac{1}{x_{n+1}^2}\left\{Q\rho\left[\sigma+ \alpha |\nabla \rho|^2-(\alpha(\rho+x\cdot \nabla \rho)+2b)x\cdot \nabla \rho \right]-\alpha\rho-2b\right \}\\
&=\dfrac{1}{x_{n+1}^2}\left\{Q\rho F-\alpha\rho-2b\right\}.
\end{align*}
}
The coefficient of $\na \rho \otimes \na \rho$ is 
\begin{eqnarray}\nonumber
\Lambda
&:=&
-Q\sigma+\alpha
-
Q
\Big[\alpha|\nabla\rho|^2-\left(\alpha(\rho+x\cdot\nabla\rho)+2b\right)x\cdot \nabla\rho\Big]\\\nonumber
&=&
\alpha
-Q\Big[\sigma + \alpha|\nabla\rho|^2-\left(\alpha(\rho+x\cdot\nabla\rho)+2b\right)x\cdot \nabla\rho\Big]
\\\nonumber
&=&
\alpha-QF.
\end{eqnarray}
Therefore \eqref{def-B} follows.
\end{proof}

\subsection{The equation at $x=0$}\label{x=0}
We are now ready to prove Theorem \ref{thm:MA at 0}.

The value of the variables in the notation list at $0$ are summarized in the following list:
\begin{align*}
a(0,\rho,\nabla \rho)&=\rho^2+|\nabla \rho|^2,\\
q(0,\rho,\nabla \rho)&=\sqrt{\rho^2+(1-\kappa^2)|\nabla \rho|^2},\\
b(0,\rho,\nabla \rho)&=\dfrac{\kappa^2-1}{\kappa \rho+q},\\
\alpha(0,\rho,\nabla \rho)&=-\dfrac{b^2}{\kappa^2-1}\left[\dfrac{\kappa q+\rho}{q}\right],\\
\beta(0,\rho,\nabla \rho)&=\dfrac{b^2}{q},\\
Q(0,\rho,\nabla \rho)&=\dfrac{\kappa^2-1}{\rho(\kappa q+\rho)},\\
\sigma(0,\rho,\nabla \rho)&=\kappa-b\rho=\dfrac{\kappa q+\rho}{\kappa \rho+q},\\
\gamma(0,\rho,\nabla \rho) &=b-Q\,b|\nabla \rho|^2,\\
F(0,\rho,\nabla \rho)&=\sigma+\alpha |\nabla \rho|^2.\\
\end{align*}

\begin{proof}[Proof of Theorem \ref{thm:MA at 0}]

Notice that at $x=0$
\begin{eqnarray}\nonumber
\alpha-QF
&=&
\alpha-\frac{\kappa^2-1}{\rho(\kappa q+\rho)}[\kappa-b\rho+\alpha|\na\rho|^2]
=
\alpha\left(1-\frac{\kappa^2-1}{\rho(\kappa q+\rho)}|\na\rho|^2\right)
-\frac{\kappa^2-1}{\rho(\kappa q+\rho)}\sigma
\\\nonumber
&=&
\alpha\frac{(\rho^2+(1-\kappa^2)|\na \rho|^2)+\rho\kappa q}{\rho(\kappa q+\rho)}
-\frac{\kappa^2-1}{\rho(\kappa\rho+ q)}
=
\alpha\frac{q^2+\rho\kappa q}{\rho(\kappa q+\rho)}
-\dfrac{b}{\rho}\\\nonumber
&=&
\dfrac{\alpha q(\kappa\rho+q)}{(\kappa q+\rho)\rho}-\dfrac{b}{\rho}
=
\dfrac{b^2(\kappa\rho+q)}{\rho(1-\kappa^2)}-\dfrac{b}{\rho}\\\nonumber
&=&\dfrac{b(\kappa^2-1)}{\rho(1-\kappa^2)}-\dfrac{b}{\rho}
=-\dfrac{2b}{\rho}.\nonumber
\end{eqnarray}
Plugging in the obtained results at $x=0$ in the formulae for $\mathcal A$, and $\mathcal B$ obtained in Lemmas \ref{lm:simplified mathcal A} and \ref{lm:simplified B}, the proof follows.
\end{proof}

\section{The A3  condition}\label{sec:MTW}
To express the A3 condition at $x=0$ we use the dummy variables $v\in \R$ $p=(p_1,\cdots,p_n)\in \R^n$ in place of $\rho$ and $\nabla \rho$. 
Recall that the $A3$ conditions requires 
\begin{equation}\label{venedikt-A3}
\sum_{i,j,l,m} \dfrac{\partial^2 \left(I_{ij}+II_{ij}\right)}{\partial p_l\partial p_m}\xi_i\xi_j\eta_{l}\eta_m\leq -c_0|\xi|^2|\eta|^2.
\end{equation}
For simplicity we will drop the dependence of the functions involved evaluated at $(0,v,p)$ and use their values obtained in Section \ref{x=0}. 

We define
\begin{align}\label{eq:I}
I(0,v,p)&=\dfrac{1}{b(0,v,p) t}\left(v Id-\dfrac{\kappa^2-1}{v}p\otimes p\right),\\ \label{eq:II}
II(0,v,p)&=\dfrac{\kappa-b(0,v,p)\rho}{b(0,v,p)}Id-\dfrac{2}{v}p\otimes p,
\end{align} 
then \eqref{eq:PDE at 0} can be written as follows
$$\det(D^2\rho+I(0,\rho,\nabla \rho)+II(0,\rho,\nabla \rho))=RHS.$$

We calculate in this section the derivatives of the matrices $I$, $II$ and the stretch function $t$ and deduce the required $A3$ condition \eqref{venedikt-A3}.

\subsection{Derivatives of I and II}
In this section, we show the following Proposition.
 \begin{proposition}\label{prop:second derivative of I}

Given orthogonal vectors $\xi$ and $\eta$ in $\R^n$, we have at $x=0$
$$
\sum\dfrac{\partial^2 \left(I_{ij}+II_{ij}\right)}{\partial p_l\partial p_m}\xi_i\xi_j\eta_{l}\eta_m=\sum_{l,m} H_{lm}\eta_l\eta_m, 
$$
with
{\small
\begin{align}\label{eq:form of H_{lm}}
\hspace{-1cm}H_{lm}&=\left(\dfrac{v|\xi|^2}{\kappa^2-1}-\dfrac{(p\cdot\xi)^2}{v}\right)\left(\dfrac{kv+q}{t}\right)_{p_lp_m}+\dfrac{\kappa|\xi|^2}{\kappa^2-1}q_{p_lp_m}\\\nonumber
&=\left(\dfrac{v|\xi|^2}{\kappa^2-1}-\dfrac{(p\cdot\xi)^2}{v}\right)\left((\kappa v+q)\left(\dfrac{1}{t}\right)_{p_lp_m}+q_{p_m}\left(\dfrac{1}{t}\right)_{p_l}+q_{p_l}\left(\dfrac{1}{t}\right)_{p_m}+q_{p_lp_m}\left(\dfrac{1}{t}\right)\right)+\dfrac{\kappa|\xi|^2}{\kappa^2-1}q_{p_lp_m}
\end{align}
}
where $t,q$ and its derivatives with respect to $p$ are evaluated at $(0,v,p)$.
\end{proposition}

\begin{proof}
We have 
$\dfrac{1}{b}=\dfrac{\kappa v+q}{\kappa^2-1}$, then
$$I_{ij}(0,v,p)=\dfrac{\kappa v+q}{\kappa^2-1}\left(\dfrac{1}{t}\right)\left(v\delta_{ij}-\dfrac{\kappa^2-1}{v}p_ip_j\right).$$
Then it follows that 
\begin{align}\label{eq:first derivative of I}
\dfrac{\partial I_{ij}}{\partial p_l}(0,v,p)=-\dfrac{\kappa v+q}{v}\left(\dfrac{1}{t}\right)(p_ip_j)_{p_l}+\left(\dfrac{\kappa v+q}{\kappa^2-1}\left(\dfrac{1}{t}\right)_{p_l}+\dfrac{q_{p_l}}{\kappa^2-1}\left(\dfrac{1}{t}\right)\right)\left(v\delta_{ij}-\dfrac{\kappa^2-1}{v}p_ip_j\right).
\end{align}
Differentiating \eqref{eq:first derivative of I} with respect to $p_m$ we get
\begin{align}\label{eq:second derivative of I}
\dfrac{\partial^2 I_{ij}}{\partial p_m\partial p_l}&=-\dfrac{\kappa v+q}{v}\left(\dfrac{1}{t}\right)(p_ip_j)_{p_lp_m}-\dfrac{q_{p_m}}{v}\left(\dfrac{1}{t}\right)(p_ip_j)_{p_l}-\dfrac{\kappa v+q}{v}\left(\dfrac{1}{t}\right)_{p_m}(p_ip_j)_{p_l}-\left(\dfrac{\kappa v+q}{v}\left(\dfrac{1}{t}\right)_{p_l}+\dfrac{q_{p_l}}{v}\left(\dfrac{1}{t}\right)\right)\left(p_ip_j\right)_{p_m}\\\nonumber
&+\quad\dfrac{1}{\kappa^2-1}\left(\dfrac{\kappa v+q}{t}\right)_{p_lp_m}\left(v\delta_{ij}-\dfrac{\kappa^2-1}{v}p_ip_j\right).
\end{align}

Notice that since $\xi$ and $\eta$ are orthogonal then the following three identities hold
\begin{align}\label{eq:orthogonality one}
\sum_{i,j,l,m=1}^n (p_ip_j)_{p_lp_m}\xi_i\xi_j\eta_l\eta_m&=\sum_{i,j,l,m=1}^n (\delta_{j m}\delta_{il}+\delta_{im}\delta_{jl})\xi_i\xi_j\eta_l\eta_m=2\sum_{i,j=1}^n\xi_i\eta_i\xi_j\eta_j= 2(\xi\cdot\eta)^2=0,\\ \label{eq:orthogonality two}
\sum_{i,j,l=1}^n (p_ip_j)_{p_l}\xi_i\xi_j\eta_l\eta_m&=\eta_m\sum_{i,j,l=1}^n (p_j\delta_{il}+p_i\delta_{jl})\xi_i\xi_j\eta_l=\eta_m\left(\sum_{j,l=1}^n p_j \xi_l\eta_l\xi_j+\sum_{i,l=1}^n p_i\xi_l\eta_l\xi_i\right)=2\eta_m(\xi\cdot \eta)(p\cdot\xi)=0,\\ \label{eq:orthogonality three}
\sum_{i,j,m=1}^n (p_ip_j)_{p_m}\xi_i\xi_j\eta_l\eta_m&=0.
\end{align}

Hence using the identities \eqref{eq:orthogonality one}, \eqref{eq:orthogonality  two}, \eqref{eq:orthogonality three}, and \eqref{eq:second derivative of I} we get 
\begin{equation}\label{eq:derivative of I}
\sum_{i,j,l,m}\dfrac{\partial^2I_{ij}}{\partial p_l\partial p_m}\xi_i\xi_j\eta_l\eta_m=\left(\dfrac{v|\xi|^2}{\kappa^2-1}-\dfrac{(p\cdot\xi)^2}{v}\right)\sum_{l.m}\left(\dfrac{kv+q}{t}\right)_{p_lp_m}\eta_l\eta_m.
\end{equation}

We next calculate the derivatives of $II$. From \eqref{eq:II} we have
$$II_{ij}=\dfrac{\kappa q(0,v,p)+\rho}{\kappa^2-1}\delta_{ij}-\dfrac{2}{v}p_ip_j.$$
Then
$$\dfrac{\partial^2 II_{ij}}{\partial p_l\partial p_m}=\dfrac{\kappa q_{p_lp_m}}{\kappa^2-1}\delta_{ij}-\dfrac{2}{v}(p_ip_j)_{p_lp_m}.$$
Therefore \eqref{eq:orthogonality one} yields 
\begin{align}\label{eq:derivative of II}
\sum_{i,j,l,m=1}^n\dfrac{\partial ^2 II_{ij}}{\partial p_lp_m}\xi_i\xi_j\eta_l\eta_m&=\dfrac{\kappa |\xi|^2}{\kappa^2-1}\sum_{l,m=1}^nq_{p_lp_m}\eta_l\eta_m.
\end{align}

The proof of Proposition \ref{prop:second derivative of I} is now complete.
 \end{proof}

\subsection{Derivatives of the stretch function} To compute the derivatives of $t$ we use a computation form \cite{K-siam}.  Recall that $\psi(Z)=0$ for every $Z\in \Sigma$ then using the dummy variables $(x,v,p)$ we have from \eqref{def-Z}
$$\psi\left(v\left(x,\sqrt{1-|x|^2}\right)+tY\right)=0,$$
where using Lemma \ref{lm:formula for Y}
\begin{equation}\label{eq:dummy Y}
Y(x,v,p)=\kappa\left(x,\sqrt{1-|x|^2}\right)+b(0,v,p)\left((p,0)-\left(x,\sqrt{1-|x|^2}\right)(v+x\cdot p)\right).
\end{equation}
At $x=0$, denoting $e_{n+1}=(0,0,\cdots,0,1)$ the $n+1$ dimensional vertical vector, we have
$$
\psi\left(ve_{n+1}+tY\right)=0.
$$

Differentiating with respect to $p_m$ yields
$$0=\sum_{k=1}^{n+1}\psi^k (t_{p_{m}}y_k+t(y_k)_{p_m}),$$
hence for $m=1,\cdots,n$
\begin{eqnarray}\label{t_k}
  \frac{t_{p_m}}{t}=-\frac{\sum_{k}\psi^k (y_k)_{p_m}}{\sum_k\psi^k y_k}=-\frac{\nabla\psi\cdot Y_{p_m}}{\nabla \psi\cdot Y}, \quad m=1, \dots n.
\end{eqnarray}
Hence
\begin{equation}\label{eq:One derivative 1/t}
\left(\frac1t\right)_{p_m}=\frac1t \frac{\nabla \psi \cdot Y_{p_m}}{\nabla \psi\cdot Y}, \quad m=1, \dots n.
\end{equation}

Differentiating \eqref{t_k} with respect to $p_l$ we get
\begin{eqnarray}\nonumber
  \frac{t_{p_mp_l}}{t}-\frac{t_{p_m}t_{p_l}}{t^2}
  &=&-\Bigg[\frac{\sum_{ks}\psi^{ks}({}(y_s)_{p_l}t+{}y_st_{p_l}){}(y_k)_{p_m}+\sum_k\psi^k{{}(y_k)_{p_mp_l}}}{\nabla\psi\cdot{}Y}\\\nonumber
  &&-\frac{\sum_k\psi^k{}(y_k)_{p_m}}{(\nabla\psi\cdot{}Y)^2}
  \left(\sum_{k,s}\psi^{ks}({}(y_k)_{p_l}t+{}y_st_{p_l}){}y_k+\sum_k\psi^k{}(y_k)_{p_l}\right)\Bigg]\\\nonumber
  &=&-\frac{1}{\nabla\psi\cdot{}Y }\bigg[
  \left(\na^2\psi{}Y_{p_l}{}\cdot Y_{p_m}-\frac{\nabla\psi{}\cdot Y_{p_m}}{\nabla\psi\cdot{}Y}\na^2\psi{}Y_{p_l}{}\cdot Y\right)t\\\nonumber
  &&+\left(\na^2\psi{}Y{}\cdot Y_{p_m}-\frac{\nabla\psi{}\cdot Y_{p_m}}{\nabla\psi\cdot{}Y}\na^2\psi{}Y\cdot{}Y\right)t_{p_l}\\\nonumber
  &&+\nabla\psi{}\cdot Y_{p_mp_l}-\frac{\nabla\psi{}\cdot Y_{p_m}}{\nabla\psi\cdot{}Y}\nabla\psi{}\cdot Y_{p_l}\bigg]\\\nonumber
  &=&-\frac{1}{\nabla\psi\cdot{}Y }\bigg[
  \left(\na^2\psi{}Y_{p_m}\cdot{}Y_{p_l}-\frac{\nabla\psi\cdot{}Y_{p_m}}{\nabla\psi\cdot{}Y}\na^2\psi{}Y\cdot{}Y_{p_l}\right)t\\\nonumber
  &&+\left(\na^2\psi{}Y_{p_m}{}\cdot Y-\frac{\nabla\psi\cdot {}Y_{p_m}}{\nabla\psi \cdot {}Y}\na^2\psi{}Y\cdot{}Y\right)t_{p_l}
  +\nabla\psi{}\cdot Y_{p_mp_l}\bigg]\\\nonumber
  &&+\frac{t_{p_m}t_{p_l}}{t^2}.
\end{eqnarray}

Therefore using the calculation above and \eqref{t_k}, and rearranging the terms we infer
\begin{eqnarray}\label{t_2nd_der}
\left(\frac1t\right)_{p_mp_l}&=&-\frac{t_{p_mp_l}}{t^2}+\frac{2t_{p_m}t_{p_l}}{t^3}=\frac{1}{t(\nabla\psi\cdot{}Y)
}\bigg[
  \left(\na^2\psi{}Y_{p_m}{}\cdot Y_{p_l}-\frac{\nabla\psi{}\cdot Y_{p_m}}{\nabla\psi\cdot{}Y}\na^2\psi{}Y\cdot{}Y_{p_l}\right)t\\\nonumber
  &&+\left(\na^2\psi{}Y_{p_m}\cdot{}Y-\frac{\nabla\psi{}\cdot Y_{p_m}}{\nabla\psi{}\cdot Y}\na^2\psi{}Y{}\cdot Y\right)t_{p_l}
  +\nabla\psi{}\cdot Y_{p_mp_l}\bigg]\\\nonumber
  &=& \dfrac{1}{t(\nabla \psi\cdot Y)}\bigg[\left(\na^2\psi Y_{p_m}\cdot Y_{p_l}+\dfrac{t_{p_m}}{t}\na^2\psi Y\cdot Y_{p_l}\right)t
  \\\nonumber
  &&+\left(\na^2\psi Y_{p_m}\cdot Y+\dfrac{t_{p_m}}{t}\na^2\psi Y\cdot Y\right)t_{p_l}+\nabla \psi\cdot Y_{p_mp_l}\bigg]\\\nonumber
  &=&\dfrac{1}{t(\nabla \psi\cdot Y)}\bigg[\frac{1}{t}\na^2\psi\left(tY_{p_m}+t_{p_m}Y\right)\cdot tY_{p_l}+\frac{1}{t}\na^2\psi\left(tY_{p_m}+t_{p_m}Y\right)\cdot (t_{p_l}Y)+\nabla \psi\cdot Y_{p_mp_l}\bigg]\\\nonumber
  &=&\frac{1}{t(\nabla\psi\cdot{}Y)
}\bigg[\frac{1}{t}\na^2\psi
Z_{p_m}\cdot Z_{p_l}+\nabla\psi{}\cdot Y_{p_mp_l}\bigg].
\end{eqnarray}

We calculate the derivative of $Y$. In fact from \eqref{eq:dummy Y}, and the formula for $b$ we have
$$Y(0,v,p)=\kappa e_{n+1}+\dfrac{\kappa^2-1}{\kappa v+q}\left((p,0)-ve_{n+1}\right).$$
Then 
\begin{equation}\label{eq:two derivatives for Y}
\left[Y(\kappa v+q)\right]_{p_mp_l}=\kappa e_{n+1}q_{p_mp_l}.
\end{equation}

\subsection{Final form of $A3$}
Replacing \eqref{eq:One derivative 1/t}, \eqref{t_2nd_der}, \eqref{eq:two derivatives for Y} in \eqref{eq:form of H_{lm}}
we get
{\small
\begin{align*}
H_{lm}&=\left(\dfrac{v|\xi|^2}{\kappa^2-1}-\dfrac{(p\cdot \xi)^2}{v}\right)\bigg[\frac{\kappa v+q}{t^2(\nabla\psi\cdot{}Y)}
\na^2\psi
Z_{p_m}\cdot Z_{p_l}+\dfrac{1}{t(\nabla\psi\cdot Y)}\left((\kappa v+q)\nabla\psi{}\cdot Y_{p_mp_l}+q_{p_m}\nabla \psi\cdot Y_{p_l}+q_{p_l}\nabla \psi\cdot Y_{p_m}+q_{p_lp_m}\nabla \psi\cdot Y\right)\bigg]\\
&\qquad\qquad +\dfrac{\kappa |\xi|^2}{1-\kappa^2}q_{p_l}q_{p_m}\\
&= \left(\dfrac{v|\xi|^2}{\kappa^2-1}-\dfrac{(p\cdot \xi)^2}{v}\right)\bigg[\frac{\kappa v+q}{t^2(\nabla\psi\cdot{}Y)}
\na^2\psi
Z_{p_m}\cdot Z_{p_l}+\dfrac{\nabla \psi}{t(\nabla\psi\cdot Y)}\cdot[Y(\kappa v+q)]_{p_lp_m}\bigg]+\dfrac{\kappa |\xi|^2}{\kappa^2-1}q_{p_lp_m}\\
&=\left(\dfrac{v|\xi|^2}{\kappa^2-1}-\dfrac{(p\cdot \xi)^2}{v}\right)\bigg[\frac{\kappa v+q}{t^2(\nabla\psi\cdot{}Y)}
\na^2\psi
Z_{p_m}\cdot Z_{p_l}+\dfrac{\kappa\psi^{n+1}}{t(\nabla\psi\cdot Y)}q_{p_lp_m}\bigg]+\dfrac{\kappa |\xi|^2}{\kappa^2-1}q_{p_lp_m}\\
&=\left(\dfrac{v|\xi|^2}{\kappa^2-1}-\dfrac{(p\cdot \xi)^2}{v}\right)\left(\frac{\kappa v+q}{t^2(\nabla\psi\cdot{}Y)}
\na^2\psi
Z_{p_m}\cdot Z_{p_l}\right)+\left[\dfrac{\kappa \psi^{n+1}}{t\nabla \psi\cdot Y}\left(\dfrac{v|\xi|^2}{\kappa^2-1}-\dfrac{(p\cdot \xi)^2}{v}\right)+\dfrac{\kappa|\xi|^2}{\kappa^2-1}\right]q_{p_lp_m}.
\end{align*}

We show that $D^2q\geq 0$ for $\kappa<1$. Recall $q(0,v,p)=\sqrt{v^2+(1-\kappa^2)|p|^2}$ so for $\kappa<1$.
\begin{align*}
\sum_{l,m}q_{p_lp_m}\eta_l\eta_m&=\sum_{l,m}\left(\dfrac{1-\kappa^2}{q}\delta_{lm}-\dfrac{(1-\kappa^2)p_l p_m}{q^2}\right)\eta_l\eta_m\\
&=\dfrac{1-\kappa^2}{q}|\eta|^2-\dfrac{(1-\kappa^2)^2}{q^3}(p\cdot\eta)^2=\dfrac{1-\kappa^2}{q}\left(|\eta|^2-\dfrac{1-\kappa^2}{v^2+(1-\kappa^2)|p|^2}(p\cdot\eta)^2\right)\\
&\geq 0.
\end{align*}

Consequently returning to $H_{lm}$ we have that
\begin{equation}\label{Parxomenko}
H_{lm}
=
\left(\dfrac{v|\xi|^2}{\kappa^2-1}-\dfrac{(p\cdot \xi)^2}{v}\right)\left(\frac{\kappa v+q}{t^2(\nabla\psi\cdot{}Y)}
\na^2\psi
Z_{p_m}\cdot Z_{p_l}\right)
+
\left[\dfrac{\kappa \psi^{n+1}}{t\nabla \psi\cdot Y}\left(\dfrac{v|\xi|^2}{\kappa^2-1}-\dfrac{(p\cdot \xi)^2}{v}\right)+\dfrac{\kappa|\xi|^2}{\kappa^2-1}\right]q_{p_lp_m},
\end{equation}
where  the second term is nonpositive for $\kappa<1$. Hence, if the first term is negative definite then the MTW condition follows. 

\subsection{Proof of Theorem \ref{thm-2}} We bring the second fundamental form of receiver $\Sigma$ into the play 
Let $Z_0$ be a fixed point on $\Sigma$.
Introduce a new coordinate
system $\hat x_1, \dots, \hat x_n,
\hat x_{n+1}$ near $Z_0$, with $\hat x_{n+1}$  pointing in  $Y$ direction.
Since we require  $\nabla \psi\ne 0$, without loss of
generality we assume that near $Z_0$, in $\hat x_1, \dots, \hat x_n,
\hat x_{n+1}$ coordinate system $\Sigma$ has a representation
$\hat x_{n+1}=\phi(\hat x_1, \dots, \hat x_n).$
Recall that the second fundamental form of $\Sigma $ is
\begin{eqnarray}\label{second-form}
\mbox{II}=\frac{\p^2_{\hat x_i, \hat x_j}\phi}{\sqrt{1+|D\phi|^2}} , \qquad i,j=1,\dots, n, 
\end{eqnarray}
if we choose the normal of $\Sigma$ at $Z_0$ to be
$\frac{(-D_{\hat x_1}\phi, \dots, -D_{\hat x_n}\phi, 1)}{\sqrt{1+|D\phi|^2}}, D\phi=(D_{\hat x_1}\phi, \dots, D_{\hat x_n}\phi,0)$.

Denote  $\psi(Z)=Z^{n+1}-\phi(z)$ and assume that
near $Z_0$, $\Sigma $ is given by the equation $\psi=0$. It follows that
\begin{eqnarray}\label{Hess-psi}
 \nabla^2\psi= -\left|\begin{array}{cccc}
        \phi^{11} &\cdots & \phi^{1n} & 0\\
         \vdots  &\ddots &\vdots &\vdots \\
        \phi^{n1} &\cdots & \phi^{nn} & 0\\
        0 &\cdots & 0 & 0\\
         \end{array} \right|.
\end{eqnarray}

Therefore for  $Z=x+u e_{n+1}+tY$ we have $\nabla^2\psi  Y=0$ and hence
\begin{eqnarray}\label{2nd-f}
  \nabla^2\psi Z_{p_m}Z_{p_l}&=&\nabla^2\psi(tY_{p_m}+t_{p_m}Y)(tY_{p_l}+t_{p_l}Y)\\\nonumber
  &=&t^2\nabla^2\psi Y_{p_m}Y_{p_l}\\\nonumber
  &=&-t^2\nabla^2\phi  Y_{p_m}Y_{p_l}.\\\nonumber
\end{eqnarray}

We have $Y(\kappa\rho+q)=\kappa X (\kappa\rho+q)+ (1-\kappa^2)(e_{n+1}\rho-\na \rho)$.
Differentiating this equality with respect to  $p_k$ we infer
\[
(Y(\kappa\rho+q))_{p_k}=(\kappa X (\kappa\rho+q)+ (1-\kappa^2)(e_{n+1}\rho-p))_{p_k}. 
\]
Therefore
\begin{eqnarray}
Y_{p_k}
&=&
\frac1{\kappa \rho+q}\left[(\kappa X-Y)q_{p_k}-(1-\kappa^2)e_k\right]\\
&=&
\frac{1-\kappa^2}{\kappa \rho+q}\left[\frac{p_k}q(\kappa X-Y)-e_k\right].
\end{eqnarray}
Since $\na^2 \psi Y=0$ then it is enough to consider the projection of the vector 
$\zeta=\frac{1-\kappa^2}{\kappa \rho+q}\left[\frac{p_k}q\kappa X-e_k\right]
$ on the hyperplane $\Pi_0$ parring through $Z_0$ and perpendicular to $Y$. 
Denote $X^0, e_k^0, k=1, \dots, n$ be the projections of $X, e_k$ on $\Pi_0$. Recall that here 
$X$ is $e_{n+1}$.
Then we have that 
\begin{eqnarray*}
-\frac1{t^2}\nabla^2\psi Z_{p_m}Z_{p_l}
&=&
\sum_{l.m}\na^2 \phi Y_{p_k}Y_{p_l}\eta_k \eta_l \\
&=&
\left(\frac{1-\kappa^2}{\kappa \rho+q}\right)^2
\nabla^2\phi
\left(\frac{p_k}q\kappa X^0-e_k^0\right)
\left(\frac{p_l}q\kappa X^0-e_l^0\right)
\eta_k\eta_l\\
&=&
\left(\frac{1-\kappa^2}{\kappa \rho+q}\right)^2
\nabla^2\phi
\left(\frac{(p\cdot \eta)}q\kappa X^0-\eta^0\right)
\left(\frac{(p\cdot \eta)}q\kappa X^0-\eta^0\right).
\end{eqnarray*}

Plugging this into \eqref{Parxomenko} with dummy variable $v, p$ we get 
{\small
\begin{eqnarray}\label{eq:Final form}
\sum_{l,m}H_{lm}\eta_l\eta_m
&=&
\left(\dfrac{v|\xi|^2}{1-\kappa^2}+\dfrac{(p\cdot \xi)^2}{v}\right)
\left(\frac{(1-\kappa^2)^2}{(\nabla\psi\cdot{}Y)(\kappa v+q)}
\nabla^2\phi
\left(\frac{(p\cdot \eta)}q\kappa X^0-\eta^0\right)
\left(\frac{(p\cdot \eta)}q\kappa X^0-\eta^0\right)\right)\\\notag
&+&\left[\dfrac{\kappa \psi^{n+1}}{t\nabla \psi\cdot Y}\left(\dfrac{v|\xi|^2}{\kappa^2-1}-\dfrac{(p\cdot \xi)^2}{v}\right)+\dfrac{\kappa|\xi|^2}{\kappa^2-1}\right]q_{p_lp_m}.
\end{eqnarray}
}
Clearly, if $\psi$ is concave in $Y$ direction (see Figure \ref{fig1}) then $H$ is a negative definite 
quadratic form obtaining hence our A3 condition.
\qed
\begin{remark}\label{rem:zagalla}
If $\kappa>1$ then we demand $|\nabla \rho|^2\le \frac{\rho^2}{\kappa^2-1}$ so that 
$b$ is well defined. Then the first factor  in \eqref{eq:Final form} is negative. However, if $\nabla^2\phi$ is negative definite too then   
the first term is positive. As for the second term in \eqref{eq:Final form} then we have that $q_{p_lp_m}$ is negative definite.
Thus if $\Sigma$ is strictly concave in $Y$ direction so that the second negative term can be absorbed then we have that $H>0$ near $0$
and the A3 condition will follow.  
\end{remark}

\section{Existence  of weak solutions}\label{sec:Regularity}

\subsection{Existence of Aleksandrov solution}\label{sec:Aleksandrov}
In \cite{Gut-Hua} Gutierrez and Huang used cartesian ovals to construct  weak solutions to the near field refractor problem. Given a point $P$ in a medium $n_2$ and a point $O$ in medium $n_1$, with $n_1\leq n_2$, the Cartesian oval $\mathcal O(P,b)$, $b>0$ is the set of points $M\in \R^{n+1}$ such that
$$|M|+\frac 1\kappa |M-P|=b.$$
Recall $\kappa=\dfrac{n_1}{n_2}$, and we consider the case when $\kappa<1$, see Figure \ref{fig1}. 

If $|P|<b<\dfrac{1}{\kappa}|P|$  we define then the lower part of the oval defined as follows
$$\mathcal O^{-}(P,b)=\left\{M\in O(P,b):\dfrac{M}{|M|}\cdot P\geq \kappa^2\left(b+\sqrt{\left(\dfrac{1}{\kappa^2}-1\right)\left(\dfrac{1}{\kappa^2}|P|^2-b^2\right)}\right)\right\}.$$
Assume $\mathcal O^-(P,b)$ separates the media $n_1$ and $n_2$. It is proved in \cite{Gut-Hua}, 
that if $X\in \S^n$ such that 
\[
X\cdot P\geq \kappa^2\left(b+\sqrt{\left(\dfrac{1}{\kappa^2}-1\right)\left(\dfrac{1}{\kappa^2}|P|^2-b^2\right)}\right),
\]
then the ray emitted from $O$ with direction $X$ is refracted by $\mathcal O^-(P,b)$ into $P$. $\Ov^{-}$ can be given radially by the function
$$h(X,P,b)=\dfrac{\frac{1}{\kappa^2}X\cdot P-b-\sqrt{\left(\dfrac{1}{\kappa^2}X\cdot P-b\right)^2-(\dfrac{1}{\kappa^2}-1)(\dfrac{1}{\kappa^2}|P|^2-b^2)}}{\dfrac{1}{\kappa^2}-1}.$$

Using these ovals the authors in \cite{Gut-Hua} showed the existence of a Brenier solution to the near field refractor problem, in case the 
target and the source satisfy some visibility conditions, see Section 6 in the mentioned paper. A solution $\Gamma=\{\rho(x)X, \ X\in \U\}$ is 
constructed such that at every point it is supported by a refracting oval $\mathcal O^-(P,b)$ for some $b>0$, and $P\in \V$,  that is 
for every $X_0\in \U$ there exists $P\in \V$ and $b>0$ such that
$$\rho(x)+\dfrac{1}{\kappa}|\rho(x)X-P|\leq P \text{\,\,\,\,\,\,for all $X\in \U$ with equality at $X=X_0$}.$$
The first visibility condition requires 
\begin{equation}\label{bychkova1}
X\cdot \dfrac{P}{|P|}\geq \kappa+\tau\,  \mbox{for some}\,  0<\tau<1-\kappa, 
\end{equation}
this to ensure that the ray touches $\mathcal O^{-}$ the refracting part of the oval, at a positive angle,  i.e. not tangentially. The second visibility condition requires that 
\begin{equation}\label{bychkova2}
\mbox{in a small cone} \, Q_{r_0}\, \mbox{with vertex}\,  O, \, \mbox{every line passing through} \, O\,  \mbox{intersects}\,  \V\,  \mbox{at at most one point}.
\end{equation} 

If $\U$, and $\V$ satisfy the above visibility conditions, and $f\in L^1(\U)$, $g\in L^1(\V)$ positive then the authors showed that if $M_0\in Q_{r_0}$ there exists a refractor as described above passing through $M_0$ such that for every Borel set $G\in \V$, 
$$\int_{\mathcal R^{-1}(G)} f=\int_G g,$$
where $\mathcal R$ is the imaging map induced by the refractor $\rho$. Moreover, it is shown that the solution has a positive distance from the source $O$ and the target $\V$.

We show in this section that every Brenier solution is an Aleksandrov solution by defining a Legendre type transform. Recall that $\V\subseteq \Sigma$, with $\Sigma$ given implicitly by $\psi(Z)=z_{n+1}-\varphi(z)=0$.

\begin{theorem}\label{thm:Aleks}
If $\U$ and $\V$ satisfies the visibility conditions \eqref{bychkova1}, \eqref{bychkova2} and $\rho$ is a Brenier solution to the near field refractor problem then $\rho$ is an Aleksandrov type solution, i.e. for every Borel set $E\in \U$ we have
$$\int_{E}f=\int_{\mathcal R(E)}g.$$
\end{theorem}
\begin{proof}
Recall that $\V\subseteq \Sigma$ with $\Sigma$ given implicitly by the function $\psi$ satisfying conditions $\H1,\H2$. Let $Z=(z,z_{n+1})\in \V$. Define the Legendre type transform
\begin{equation}\label{b-def-max}
b(Z)=\max_{X=(x,x_{n+1})\in \overline \U}( \rho(x)+|\rho(x)X-Z|).
\end{equation}
Let $X_0$ be a point where the maximum is attained then for every $X\in \U$
$$\rho(x)+\dfrac{1}{\kappa}|\rho(x)X-Z|\leq b(Z),$$
with equality at $X=x_0$, and hence the oval $\Ov(Z,b)$, supports $\rho$ at $\rho(x_0)X_0$. Notice that from the visibility condition \eqref{bychkova1} we have that $X_0\cdot \dfrac{Z}{|Z|}\geq \kappa+\tau$, and so the refracting part $\mathcal O^-(Z,b)$ is supporting $\rho$ i.e. the ray $\rho(x_0)X_0$ is refracted into $Z$.

We claim that $b$ is differentiable almost everywhere as a function of $z$. In fact,  since the solution $\rho$ is uniformly away from $\V$ then $b(Z)$ is the minimum of $C^2$ functions $\rho(x)+\sqrt{(\rho(x)x-z)^2+(\rho x_{n+1}-\phi(z))^2}$ of $z$ (for $x$ fixed) with bounded Hessian. Thus $b(Z)$ is semi-concave and so locally Lipschitz, and hence differentiable almost everywhere.

Given the set $$\mathcal S=\{Z\in \V: Z\in \mathcal R(X_1)\cap \mathcal R(X_2): \text{for $X_1\neq X_2$ in $\U$}\}.$$
We claim that if $Z\in \mathcal S$ then $Z$ is a singular point for the function $b(Z)$ and hence $|S|=0$. 
In fact, given $Z\in \mathcal S$ then there exists $b_1,b_2\in \R$ such that the oval $\mathcal O(Z,b_1)$ supports $\mathcal S$ at $X_1$, and $\mathcal O(Z,b_2)$ supports $\mathcal S$ at $X_2$. Notice that if $b_1<b_2$ then 
$$\rho(x_2)+\dfrac{1}{\kappa}|\rho(x_2)X_2-Z|=b_2>b_1,$$ this is a contradiction since $\mathcal O(Z,b_1)$ supports $\rho$. So $b_1=b_2:=b$. 
Hence the maximum 
in \eqref{b-def-max} (i.e. $b(Z)$) is attained at $X_1$ and $X_2$, and so
$$b(Z)=\rho(x_1)+\dfrac{1}{\kappa}|\rho(x_1)X_1-Z|=\rho(x_2)+\dfrac{1}{\kappa}|\rho(x_2)X_2-Z|.$$
If $Z$ is a differentiability point for $b$, then by the supporting property of the ovals we get
$$D_zb=D_z \left[\rho(x_1)+\dfrac{1}{\kappa}|\rho(x_1)X_1-Z|\right]=D_z\left[\rho(x_2)+\dfrac{1}{\kappa}|\rho(x_2)X_2-Z|\right].$$
Recall that $Z=(z,\varphi(z))$, $X_1=(x_1,x_1^{n+1})$, $X_2=(x_2,x_2^{n+1})$ then above identity implies
$$\dfrac{1}{\kappa}\dfrac{z-\rho(x_1)x_1+\nabla \phi\left[\phi(z)-x_1^{n+1}-\phi(z)\right]}{|\rho(x_1)X_1-Z|}=\dfrac{1}{\kappa}\dfrac{z-\rho(x_1)x_1+\nabla \phi\left[\phi(z)-x_1^{n+1}-\phi(z)\right]}{|\rho(x_1)X_1-Z|},$$
recall that $t(x)=|Z-\rho(x)X|$ and $Z-\rho(x)X=t(x)Y$. Simplifying above identity we get
$$\dfrac{t(x_1)y(x_1)+\nabla \phi(z) t(x_1)y_1^{n+1}}{t(x_1)}=\dfrac{t(x_2)y_2+\nabla \phi(z) t(x_2)y_1^{n+1}}{t(x_2)}.$$
Thus
\[
y(x_1)-y(x_2)+\nabla \phi(z)(y_{n+1}(x_1)-y_{n+1}(x_2))=0
\]
and so
$$Y(x_1)-Y(x_2)=(-\nabla \phi,1)(y_{n+1}(x_1)-y_{n+1}(x_2)).$$
Hence $Y(x_1)-Y(x_2)$ is orthogonal to $\na \psi$. Let $\tau$ be a vector tangent to $\psi$, hence  we have
$Y(x_1)\cdot \tau=Y(x_2)\cdot \tau$, and $Y(x)\cdot \nabla \psi$ is positive for all $x$ (by $\H3$)and so 
$Y(x_1)=Y(x_2)$. Then the line passing through $Z$ with direction $Y$ intersects the refractor at two points $\rho(x_1)X_1$ and $\rho(x_2)X_2$ but these are two points on the same half oval $\mathcal O^-(Z,b)$  and so $x_1=x_2$ a contradiction, and the claim follows.

Having showed that $|\mathcal S|=0$, proceeding as in Theorem 6.7 \cite{Gut-Sab:Freeform} it follows that the Brenier solution constructed in \cite{Gut-Hua} is in fact an Aleksandrov solution.
\end{proof}

\section{Regularity of weak solutions}
\subsection{R-convexity of $\V$}\label{subsec:R convex}

\begin{definition}[Refraction cone]
Let  $\nu_1, \nu_2$ be unit vectors and set  
$\nu_{c_1, c_2}=\frac{\nu_1 c_1+\nu_2c_2}{|\nu_1 c_1+\nu_2c_2|}$
for any two constants $c_1, c_2$. Let the unit vectors $Y_{c_1, c_2}$ be determined from the 
identity 
$$X-\dfrac{n_2}{n_1} Y_{c_1,c_2}=\lambda_{\nu_{c_1, c_2}} \nu_{c_1, c_2}$$
where $X=\xi/|\xi|$
 with $\lambda_\nu=\xi\cdot \nu-\dfrac{n_2}{n_1}\sqrt{1-\left(\dfrac{n_1}{n_2}\right)^2(1-(X\cdot \nu)^2)},$.
The two dimensional cone 
\[
C_{\xi, \nu_1, \nu_2}=\{\xi+tY_{c_1, c_2}, t>0, c_1\ge 0, c_2\ge 0\}
\] 
with vertex  
$\xi$
and spanned by the vectors $Y_{c_1, c_2}$ 
 is called a refraction cone at $\xi$.
\end{definition}

\begin{definition}[R-convexity]\label{def-R-conv}
  We say that $\V\subset \Sigma$ is $R-$convex with respect to a point $\xi$ in the cone 
  $C_\U=\{tX, t>0, X\in \U\}$
  if for any two unit vectors $\nu_1, \nu_2$ the intersection $C_{\xi, \nu_1, \nu_2}\cap \V$
  is connected. If $\V$ is $R-$convex with respect to any $\xi\in C_\U$ then we simply say that
  $\V$ is $R-$convex.
\end{definition}
In particular a  geodesic ball on the  convex  surface $\Sigma$ is an example of $R-$convex $\V$.

\subsection{Local supporting function is also global}\label{sec:loc glob}
Recall that  the supporting 
ovaloid $\Ov$  is below of the refractor 
for every $X\in \U$ which follows from the discussion in Section \ref{sec:Regularity}. 
Consequently, one may wonder if the locally admissible refractors (i.e. $\Ov$ contains  the refractor  only in a vicinity of the contact point) are also globally contained in $\Ov$.   
This issue was addressed by G. Loeper in \cite{Loep} for the optimal transfer problems. For the refractor problem we have  


\begin{lemma}\label{lem:loc-glob}
Under the condition \eqref{def-MTW} a local supporting ovaloid is also global.
\end{lemma}

\begin{proof}
The proof is very similar to that of in \cite{Loep}, \cite{Trud-Wang-strc}.
Let $\Ov_i=\Ov(Z_i, b_i), i=1, 2$ be two global supporting ovaloids of $\rho$ at $M_0$ such that 
the contact set $\Lambda\not=\{M_0\}$. Thus $\rho$ is not differentiable at $M_0$. To fix the ideas take $M_0=e_{n+1}\rho(0)$.

We claim that if $\gamma_i$ is the normal of the graph of $\Ov_i, i=1,2$ at $M_0$
then for any   $\theta\in(0, 1)$ there is $Z_{\theta}\in \Sigma\cap \mathcal C_{M_0, \gamma_1,\gamma_2}$ and $b_{\theta}>0$
such that  $\Ov_\theta=\Ov(Z_{\theta}, b_{\theta})$  is a local supporting ovaloid 
of $\rho$ at $0$ and 
\begin{equation}\label{blya-2}
D\rho_{\theta}(0)=\theta D\rho_1(0)+(1-\theta)D\rho_2(0), 
\end{equation} 
where $\rho_\theta, \rho_1, \rho_2$ are the radial functions of the 
ovaloids $\Ov_\theta, \Ov_1, \Ov_2$, respectively. 
Notice that since the radial functions $\rho_i, i=1,2$ of ovaloids are smooth then the slope of the tangent plane 
of $\rho_\theta$ is the linear combination of that of $\rho_1$ and $\rho_2$.

Moreover, the correspondence $\theta\mapsto Z_\theta$ is one-to-one thanks to the 
assumption $\nabla \psi\cdot Y>0$. 
By choosing a suitable coordinate system we can assume that $D\rho_1(0)-D\rho_2(0)=(0, \dots, 0, \alpha)$ for some 
$\alpha>0$. 
Then  we have that for all $0<\theta<1$
\begin{eqnarray}\label{blya-3}
\max[\rho_1(x), \rho_2(x)]&\ge &\theta \rho_1(x)+(1-\theta)\rho_2(x)\\\nonumber
&=&\rho(0) + \left[D\rho_2(0)+\alpha\theta\right]x_n+\frac12\left[\theta D^2\rho_1(0)+(1-\theta)D^2 \rho_2(0)\right]{x\otimes x}+o(|x|^2)
\end{eqnarray}
where the last line follows from Taylor's expansion.
 
Using the notations of Section \ref{sec:Regularity} we have that 
\begin{equation}\label{kiselev}
D^2\rho_\theta(0)=-H(0, \rho(0), p_1+\theta(p_2-p_1))
\end{equation}
where $H$ is the matrix \eqref{eq:form of H_{lm}} (see Section \ref{sec:MTW}), and we set $p_i=D\rho _i(0), i=1,2$ and used \eqref{blya-2}. For all unit vectors $\tau$ perpendicular to $x_n$ axis we have from \eqref{kiselev}
\begin{eqnarray}
\frac{d^2}{d\theta^2}D^2_{\tau\tau}\rho_\theta(0)
&=&
-\frac{d^2}{d\theta^2}H_{ij}(0, \rho(0), p_1+\theta(p_2-p_1))\tau_i\tau_j
\\\nonumber
&=&
-\alpha^2\frac{\partial^2}{\partial p_n^2} H_{jj}(0, u(0), p_1+\theta(p_2-p_1))\tau_i\tau_j
\\\nonumber
&\ge&
\alpha^2 c_0
\end{eqnarray}
where the last line follows from \eqref{def-MTW} with $c_0>0$.

Therefore
\begin{equation}\label{blya-4}
D^2_{\tau\tau}\rho_\theta(0)\le\theta D^2_{\tau\tau}H_1(0)+(1-\theta)D^2_{\tau\tau}H_2(0)-\widehat{ c_0}\theta(1-\theta)|p_1-p_2|^2
\end{equation}
where $\widehat{c_0}$ depends on $c_0$.

Differentiating $\rho(x)+\frac1{\kappa}\sqrt{(\rho(x)x-z)^2+(\rho(x)x_{n+1}-\phi(z))^2}$
in $x_i, 1\le i\le n$ we get 
\begin{equation}
\p_i\rho+\frac1{t\kappa}
\left\{
(\rho(x)x_m-z_m)(\p_i\rho(x)x_m+\rho(x)\delta_{im})+(\rho(x)x_{n+1}-\phi(z)(\p_i\rho(x)x_{n+1}-\rho\frac{x}{x_{n+1}})
\right\}=0
\end{equation}
or equivalently 
\begin{equation}
\na \rho+\frac1{t\kappa}
\left\{
(x\rho-z)\rho+[(x\rho-z)\cdot x]\na \rho+\na \rho(\rho x_{n+1}-\phi(z))x_{n+1}-\rho \frac x{x_{n+1}}(\rho x_{n+1}-\phi(z))
\right\}=0
\end{equation}
Simplifying this expressions and solving with respect to $\na \rho$ we find that 
\[
\na \rho\left(1+\frac1{t\kappa}(\rho-Z\cdot X)\right)=-\frac1{t\kappa}(x\frac\phi{x_{n+1}}-z)\rho
\]
and therefore we conclude that 
\[
\na\rho=-\frac{\rho}{\rho-Z\cdot X+t\kappa}\left(x\frac\phi{x_{n+1}}-z\right).
\]
Observe that at $x=0$ this formula yields that  
\begin{equation}
\theta p_1+(1-\theta)p_2
= \frac{\theta\rho_1(0)}{\rho_1(0)-\phi(z_1)+t\kappa}z_1
+
\frac{(1-\theta)\rho_2(0)}{\rho_2(0)-\phi(z_2)+t\kappa}z_2, 
\end{equation}
where we assume that $\Sigma$ is the graph of a function $\phi$  such that 
for $\psi(Z)=Z^{n+1}-\phi(z)$ the form  \eqref{eq:Final form} is negative definite.
From these $n+1$ equations we see that $z_\theta$ and $b_\theta=b(z_\theta)$ are smooth functions of $\theta p_1+(1-\theta)p_2$.
This yields  the following crude estimate for the remaining second order derivatives
\begin{eqnarray}\label{blya-5}
\left|\theta D_{jn}^2H_1(0)+(1-\theta)D_{jn}^2H_2(0)-D^2_{jn}H(0)\right|\le C\theta(1-\theta)|p_1-p_2|^2,\quad  j=1, \dots, n, 
\end{eqnarray}
where $C$ depends of $C^2$ form of $\phi$.
Consequently, after plugging \eqref{blya-4} and \eqref{blya-5} into \eqref{blya-3} and recalling that 
$|p_1-p_2|=\alpha$ we conclude 

\medskip
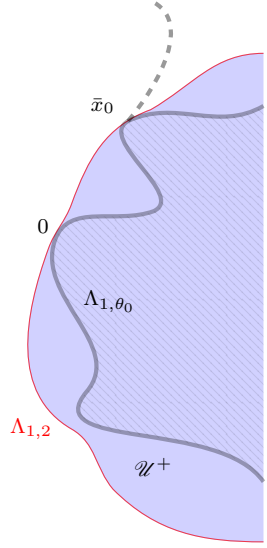
\begin{figure}[h]
\begin{center}
 \begin{tikzpicture}

\filldraw[blue!60!,opacity=.3] (3,3) to[out=180,in=30] (1.5,2.25) 
to[out=200,in=70] (0.45,1) to[out=250,in=70] (0.14,0.4)
to[out=-110,in=150] (0.45, -2)to[out=-30,in=130] (1, -2.8)to[out=-45,in=180]
(3, -3.5);

\draw[red,opacity=.7] (3,3) to[out=180,in=30] (1.5,2.25) 
to[out=200,in=70] (0.45,1) to[out=250,in=70] (0.14,0.4)
to[out=-110,in=150] (0.45, -2)to[out=-30,in=130] (1, -2.8)to[out=-45,in=180]
(3, -3.5);

\filldraw[black,opacity=.3, ultra thick, pattern=north west lines] 
(3,2.3)
to[out=210, in=25]
(1.2, 2.1)
to[out=210, in=25]
(1.54,0.9)
to[out=200,in=70]
(0.23, 0.52)
to[out=250,in=50] 
(0.65, -1.5)to[out=230, in=125](3, -2.7);

\draw[ultra thick, black, opacity=.4, dashed]
(1.2,2.1)to[out=50, in=-25](1.5,3.7);

\node[left] at (0.3, -2) {\textcolor{red}{$\Lambda_{1,2}$}};
\node[right] at (0.5, -0.3) {$\Lambda_{1, \theta_0}$}; 

\node[right] at (0.6, 2.3) {$\bar x_0$};
\node[left] at (0.25, 0.7) {$0$}; 
\node[left] at (1.9, -2.5) {$\U^+$}; 

\end{tikzpicture}
\end{center}
\caption{The inclusion principle: the $A3$ condition demands that after striking $\bar x_0$ the 
set $\Lambda_{1,\theta_0}$ cannot cross $\Lambda_{1,2}$, in other words $\Lambda_{1, \theta_0}$
cannot leave $\U^+$ (the blue region) and continue along the dashed surface.}\label{fig2}
\end{figure}

\begin{eqnarray}\nonumber
\max[\rho_1(x), \rho_2(x)]
&\ge & 
\rho(0)+\left[p_2+\alpha\theta\right] x_n+\frac12\left[\theta D^2\rho_1(0)+(1-\theta)D^2 \rho_2(0)\right]{x\otimes x}+o(|x|^2)\\\nonumber
&\ge&
\rho(0)+\left[p_2+\alpha\theta\right] x_n+
\frac12 D^2\rho_\theta (0){x\otimes x}+\widehat{ c_0}\theta(1-\theta)\alpha^2\sum_{j=1}^{n-1}x_j^2
- C\theta(1-\theta)\alpha^2|x_n||x|+ o(|x|^2)\\\nonumber
&=&
\rho(0)+\left[p_2+\alpha\theta\right] x_n+
\frac12 D^2\rho_\theta (0){x\otimes x}-\\\nonumber
&&+\widehat{ c_0}\theta(1-\theta)\alpha^2|x|^2
- \theta(1-\theta)\alpha^2(C|x_n||x|+\widehat{ c_0}x_n^2)+ o(|x|^2)\\\nonumber 
&\ge&\rho(0)+\left[p_2+\alpha\theta\right] x_n+
\frac12 D^2\rho_\theta (0){x\otimes x}-\\\nonumber
&&
+\frac{\widehat{ c_0}}2\theta(1-\theta)\alpha^2|x|^2
-\theta(1-\theta)\alpha^2\left(\frac{2C^2}{\widehat{c_0}}+\widehat{ c_0}\right)x_n^2+ o(|x|^2), 
\end{eqnarray}
where the last line follows from H\"older's inequality.
Now fixing $\theta_0$ as in  \eqref{blya-2} and using the estimate $|D^2\rho_\theta (0)-D^2\rho_{\theta_0} (0)|\le C|\theta-\theta_0|$
(with $C$ depending on $\phi$)
we obtain 
\begin{eqnarray}\nonumber
\max[\rho_1(x), \rho_2(x)]
&\ge &
\rho(0)+\left[p_2+\alpha\theta_0\right] x_n+
\frac12 D^2\rho_{\theta_0} (0){x\otimes x}+\\\nonumber
&&+\alpha\left[\theta-\theta_0\right] x_n+\frac12 D^2\rho_{\theta} (0){x\otimes x}-\frac12 D^2\rho_{\theta_0} (0){x\otimes x}\\\nonumber
&&+\frac{\widehat{ c_0}}2\theta(1-\theta)\alpha^2|x|^2
-\theta(1-\theta)\alpha^2\left(\frac{2C^2}{\widehat{c_0}}+\widehat{ c_0}\right)x_n^2+ o(|x|^2)\\\nonumber
&\ge &
\rho_{\theta_0} (x)+\alpha\left[\theta-\theta_0\right] x_n+C|\theta_0-\theta||x|^2-\\\nonumber
&&+\frac{\widehat{ c_0}}2\theta(1-\theta)\alpha^2|x|^2
- \alpha^2\left(\frac{2C^2}{\widehat{c_0}}+\widehat{ c_0}\right)x_n^2+ o(|x|^2).\\\nonumber
\end{eqnarray}
Choosing $\theta=\theta_0+x_n\alpha\left(\frac{2C^2}{\widehat{c_0}}+\widehat{ c_0}\right)$ such that 
$|x_n|<\delta$ with sufficiently small $\delta$
we finally obtain 
\begin{eqnarray}\label{Leop-000}
\max[\rho_1(x), \rho_2(x)]
&\ge &\rho_{\theta_0} (x)-C|\theta_0-\theta||x|^2+\frac{\widehat{c_0}}2\theta(1-\theta)\alpha^2|x|^2+ o(|x|^2)=\\\nonumber
&=& 
\rho_{\theta_0} (x)+\frac{\widehat{c_0}}2\theta(1-\theta)\alpha^2|x|^2+ o(|x|^2)\\\nonumber
&\ge& 
\rho_{\theta_0} (x), \quad \forall x\in B_\delta.
\end{eqnarray}

This, in particular, implies that $\rho_{\theta_0}$ is a local supporting ovaloid near $x=0$. 

It remains to check that $\rho_{\theta_0}$ is also a global supporting ovaloid.

\begin{lemma}
There is a unit direction $e$ such that the normal images of the cones form positive angle with 
$e$. Hence there is an $n-$dimensional  plane $\Pi$ such that $\Ov_1, \Ov_2, \Ov_\theta$ are graphs over 
$\Pi$ in $e$ direction.
\end{lemma}
\begin{proof}
By construction of the weak solutions we know that the visibility cones 
of the ovaloids have solid angles $\pi-\tau$ for some 
$\tau>0$, see \eqref{bychkova1}. Suppose that $C_1$ and $C_2$ are the visibility cones of 
$\Ov_1, \Ov_2$,  respectively. The surface $\Ov^*$ given by $\rho=\max[\rho_1, \rho_2]$ is 
contained in $C_1\cap C_2$. Thus it is enough to show that 
the intersection of the ovaloids $\Ov_1, \Ov_2$ is inside $\Ov_\theta$ in 
$C:=C_1\cap C_2$, since outside of the intersection cone the inclusion is obvious. 

The Gauss map of $\Ov^*$ is contained in some hemisphere since both cones $C_1, C_2$
have apertures $<\pi-\tau$ thanks to \eqref{bychkova1}, and hence so does $C$. Moreover, the image of the Gauss map of 
$\Ov_\theta\cap C$ is contained in that of $\Ov^*$ hence there is a
direction $e$ such that $\Ov^*\cap C$ and $\Ov_\theta\cap C$ are graphs in $e$ direction
\end{proof}

From here  we can proceed as in \cite{K-siam}, Lemma 10.1.

The set $\Lambda_{1,2}=\{x\in \Pi, \ :\ \rho_1(x)=\rho_2(x)\}$ passes through $0$ and splits $\U$ into two parts $\U^+$ and $\U^-$
(recall that $\Lambda_{12}$ is a smooth surface).
It follows from \eqref{Leop-000} that the contact sets $\Lambda_{i,\theta_0}=\{x\in \R^n\ :\  \rho_i(x)=\rho_{\theta_0}(x)\}, i=1,2$
are tangent to $\Lambda_{12}$ from one side in some vicinity of $0$, say in $\U^+$. 
If there is $\bar x_0\not =0$ such that, say,  
$\bar x_0\in \Lambda_{1,2}\cap \Lambda_{1, \theta_0}$ then $\rho_1(\bar x_0)=\rho_2(\bar x_0)=\rho_{\theta_0}(\bar x_0)$
and $\na \rho_{\theta_0}(\bar x_0)=\bar \theta_0 \na \rho_1 (\bar x_0)+(1-\bar\theta_0)\na \rho_2(\bar x_0)$ with possibly 
different $\bar\theta_0$. 

Observe that by construction the ray emitted from $\bar X=(\bar x_0, \sqrt{1-|\bar x_0|^2})$  after refraction 
from $\Ov_1, \Ov_2$ and $\Ov_{\theta_0}$ hits the point $Z_1, Z_2$ and $Z_{\theta_0}$, respectively.
Then repeating the argument above with $0$ replaced by $\bar x_0$ and $\theta_0$ by $\bar\theta_0$ (but keeping $\rho_{\theta_0}$ fixed),  we can see that 
\eqref{Leop-000} is satisfied in 
$B_\delta(\bar x_0)$ implying that $ \Lambda_{1, \theta_0}$ is tangent with $ \Lambda_{1, 2}$ at $\bar x_0$ and lies in $\U^+$, see Figure \ref{fig2}.
Thus $\rho_{\theta_0}$ is a global supporting ovaloid.
\end{proof}
Now the proof of $C^{2, \alpha}$ regularity  follows as it was explained in Introduction before the statement of Theorem \ref{thm:thm4}.

\section*{
Acknowledgment
}
The authors were partially supported by Research Grant 2015/19/P/ST1/02618 from the National Science Centre,
Poland, entitled 'Variational Problems in Optical Engineering and Free Material Design'. A.S. would like to thank the University of Edinburgh that hosted him twice in order for the completion of this paper.  A.K. would like to thank the  Banach Center of the Institute of Mathematics of the Polish Academy of Sciences for hospitality. 
}
{\footnotesize This project has received funding from the European Union's Horizon 2020 research and innovation program under the Marie Sk\l{}odowska-Curie grant agreement No. 665778.}\\ \\


\begin{bibdiv}
\begin{biblist}

\bib{Gut-Abe}{article}{
   author={Abedin, Farhan},
   author={Guti\'errez, Cristian E.},
   title={An iterative method for generated Jacobian equations},
   journal={Calc. Var. Partial Differential Equations},
   volume={56},
   date={2017},
   number={4},
   pages={101},
   issn={1432-0835},
   review={\MR{3669141}}
}

\bib{Abe-Gut-Giu}{article}{
   author={Abedin, Farhan},
   author={Guti\'errez, Cristian E.},
   author={Tralli, Giulio},
   title={$C^{1,\alpha}$ estimates for the parallel refractor},
   journal={Nonlinear Anal.},
   volume={142},
   date={2016},
   number={2},
   pages={1--25},
   review={\MR{3508055}},
}

\bib{Jun-Gui}{article}{
   author={Guillen, Nestor},
   author={Kitagawa, Jun},
   title={Pointwise estimates and regularity in geometric optics and other generated Jacobian equations},
   journal={Comm. Pure Appl. Math. },
   volume={70},
   date={2017},
   number={6},
   pages={1146--1220},
   review={\MR{3639322}},
}



\bib{GM}{article}{
   author={Guti\'errez, Cristian E.},
   author={Mawi, Henok},
   title={The refractor problem with loss of energy},
   journal={Nonlinear Anal.},
   volume={82},
   date={2013},
   pages={12--46},
   issn={0362-546X},
   review={\MR{3020894}},
}

\bib{Gut-Hua}{article}{
   author={Guti\'errez, Cristian E.},
   author={Huang, Qingbo},
   title={The near field refractor},
   journal={Ann. Inst. H. Poincar\'e Anal. Non Lin\'eaire},
   volume={31},
   date={2014},
   number={4},
   pages={655--684},
   issn={0294-1449},
   review={\MR{3249808}},
}

\bib{Gut-Hua-2}{article}{
   author={Guti\'errez, Cristian E.},
   author={Huang, Qingbo},
   title={The refractor problem in reshaping light beams},
   journal={Arch. Ration. Mech. Anal.},
   volume={193},
   date={2009},
   number={2},
   pages={413--443},
   issn={1432-0673},
   review={\MR{2525122}},
}



\bib{Gut-Tou}{article}{
   author={Guti\'errez, Cristian E.},
   author={Tournier, Federico},
   title={Regularity for the near field parallel refractor and reflector problems.},
   journal={Calc. Var. Partial Differential Equations},
   volume={54},
   date={2015},
   number={1},
   pages={917--949},
   review={\MR{3385186}},
}



\bib{Gut-Sab:SIAM}{article}{
   author={Guti\'errez, Cristian E.},
   author={Sabra, Ahmad},
   title={Aspherical Lens Design and Imaging},
   journal={SIAM J. Imaging Sci.},
   volume={9},
   date={2016},
   number={1},
   pages={386--411},
   issn={1936-4954},
   review={\MR{3477314}}
}
		
\bib{Gut-Sab:Freeform}{article}{
   author={Guti\'errez, Cristian E.},
   author={Sabra, Ahmad},
   title={Freeform Lens Design for Scattering Data with General Radiant Fields},
   journal={Arch. Ration. Mech. Anal.},
   volume={228},
   date={2018},
   number={2},
   pages={341--399},
   issn={1432-0673},
   review={\MR{3766979}}
}
		

\bib{Trudinger-D}{article}{
   author={Jiang, Feida},
   author={Trudinger, Neil S.},
   author={Yang, Xiao-Ping},
   title={On the Dirichlet problem for Monge-Amp\`ere type equations},
   journal={Calc. Var. Partial Differential Equations},
   volume={49},
   date={2014},
   number={3-4},
   pages={1223--1236},
   issn={0944-2669},
   review={\MR{3168630}},
   doi={10.1007/s00526-013-0619-3},
}
\bib{Trudinger-D2}{article}{
   author={Jiang, Feida},
   author={Trudinger, Neil S.},
   title={On the Second Boundary Value Problem for Monge?Amp\`ere Type Equations and Geometric Optics},
   journal={Arch. Ration. Mech. Anal.},
   volume={229},
   date={2018},
   number={2},
   pages={547--567},
   issn={1432-0673},
   review={\MR{3803771}},
}



\bib{KW}{article}{
   author={Karakhanyan, Aram},
   author={Wang, Xu-Jia},
   title={On the reflector shape design},
   journal={J. Differential Geom.},
   volume={84},
   date={2010},
   number={3},
   pages={561--610},
   issn={0022-040X},
   review={\MR{2669365}},
}

\bib{K-siam}{article}{
   author={Karakhanyan, Aram L.},
   title={An inverse problem for the refractive surfaces with parallel
   lighting},
   journal={SIAM J. Math. Anal.},
   volume={48},
   date={2016},
   number={1},
   pages={740--784},
   issn={0036-1410},
   review={\MR{3463050}},
}



\bib{Loep}{article}{
   author={Loeper, Gr\'{e}goire},
   title={On the regularity of solutions of optimal transportation problems},
   journal={Acta Math.},
   volume={202},
   date={2009},
   number={2},
   pages={241--283},
   issn={0001-5962},
   review={\MR{2506751}},
   doi={10.1007/s11511-009-0037-8},
}


\bib{Luneburg}{book}{
   author={Luneburg, R. K.},
   title={Mathematical theory of optics},
   series={Foreword by Emil Wolf; supplementary notes by M. Herzberger},
   publisher={University of California Press, Berkeley, Calif.},
   date={1964},
   pages={xxx+448},
   review={\MR{0172589}},
}



\bib{MTW}{article}{
   author={Ma, Xi-Nan},
   author={Trudinger, Neil S.},
   author={Wang, Xu-Jia},
   title={Regularity of potential functions of the optimal transportation
   problem},
   journal={Arch. Ration. Mech. Anal.},
   volume={177},
   date={2005},
   number={2},
   pages={151--183},
   issn={0003-9527},
   review={\MR{2188047}},
}



\bib{Trudinger}{article}{
   author={Trudinger, Neil S.},
   title={On the local theory of prescribed Jacobian equations},
   journal={Discrete Contin. Dyn. Syst.},
   volume={34},
   date={2014},
   number={4},
   pages={1663--1681},
   issn={1078-0947},
   review={\MR{3121636}},
   doi={10.3934/dcds.2014.34.1663},
}

\bib{Trud-Wang-strc}{article}{
   author={Trudinger, Neil S.},
   author={Wang, Xu-Jia},
   title={On strict convexity and continuous differentiability of potential
   functions in optimal transportation},
   journal={Arch. Ration. Mech. Anal.},
   volume={192},
   date={2009},
   number={3},
   pages={403--418},
   issn={0003-9527},
   review={\MR{2505359}},
   doi={10.1007/s00205-008-0147-z},
}


\bib{Wan}{article}{
   author={Wang, Xu-Jia},
   title={On the design of a reflector antenna. II.},
   journal={Calc. Var. Partial Differential Equations},
   volume={20},
   date={2004},
   number={3},
   pages={329--341},
   issn={1432-0835},
   review={\MR{2062947}},
}


\end{biblist}
\end{bibdiv}
\end{document}